\documentclass[12pt]{amsart}
\usepackage{amsfonts}
\usepackage{mathrsfs}
\usepackage{amscd,amsmath,amssymb,amsfonts}
\usepackage{color}
\usepackage{srcltx}
\usepackage[all]{xy}
\usepackage{tikz-cd}

\theoremstyle{plain}
\newtheorem{thm}{Theorem}
\newtheorem{lem}[thm]{Lemma}
\newtheorem{cor}[thm]{Corollary}
\newtheorem{prop}[thm]{Proposition}

\theoremstyle{definition}
\newtheorem{defn}[thm]{Definition}

\newtheorem{rmk}[thm]{Remark}

\newtheorem{ex}[thm]{Example}

\numberwithin{thm}{section} \numberwithin{equation}{section}

\newcommand{\sD}{{\mathcal D}}
\newcommand{\sE}{{\mathcal E}}

\newcommand{\sL}{{\mathcal L}}

\newcommand{\sO}{{\mathcal O}}
\newcommand{\sP}{{\mathcal P}}
\newcommand{\sQ}{{\mathcal Q}}

\begin{document}

\title{Moduli spaces and the algebra of conformal blocks}

\author{Yanglong Zhang and Mingshuo Zhou}

\address{Center of Applied Mathematics, School of Mathematics, Tianjin University, No.92 Weijin Road, Tianjin 300072, P. R. China}
\email{zhangyanglong15@mails.ucas.edu.cn}

\address{Center of Applied Mathematics, School of Mathematics, Tianjin University, No.92 Weijin Road, Tianjin 300072, P. R. China}
\email{zhoumingshuo@amss.ac.cn}

\thanks{This work was supported by National Natural Science Foundation of China (Grant No. 12171352).}
\begin{abstract} For a classical simple and simply {connected group} $G$, let $\mathcal{M}_{G,\omega}$ be the moduli space of $\omega$-semistable parabolic $G$-bundles on a complex smooth projective curve of genus $g$. We prove two results in this article: (1) $\mathcal{M}_{G,\omega}$ is of Fano type when $g\geq 3$; (2) the algebra of conformal blocks on any $n$-pointed stable curve for a classical simple Lie algebra is finitely generated.
\end{abstract}
\maketitle

\section{Introduction}
Let $X$ be a normal projective variety over $\mathbb{C}$. We say that $X$ is of Fano type (or log Fano), if there exists an effective $\mathbb{Q}$-divisor $\Delta$ such that $(X,\Delta)$ is klt and $-(K_X+\Delta)$ is nef and big (see \cite{PS}). Schwede and Smith proved that a log Fano projective  variety in characteristic
zero is of globally F-regular type (see \cite{Sc}). Thus, it enjoys remarkable geometric and cohomological properties.

Let $G$ be an affine algebraic group, and ${C}$ be a smooth projective curve over $\mathbb{C}$ with a finite subset $I\subset C$. Fix a parabolic subgroup $P_x\subset G$ for each $x\in I$. By a {\em parabolic $G$-bundle}, we mean a principal $G$-bundle $E$ on $C$ with sections $\{s_x:\{x\}\rightarrow E|_{\{x\}}/P_x\}_{x\in I}$ and the given weight $a=\{a_x\}_{x\in I}$, where $a_x\in X^{\ast}(P_x)^{\vee}_{\mathbb{Q},+}$. Let $\omega=(\{P_x\},a)$ denote the parabolic data.
If $\mathrm{Lie}(a_x)\in \Phi_0^s$ (see (\ref{equation fundamental alcove})) for $x\in I$, there exists a moduli space $\mathcal{M}_{G,\omega}$ of $\omega$-semistable {\em parabolic $G$-bundles} on $C$. Our first result is

\goodbreak

\begin{thm}[=Theorem \ref{thm Fano type of moduli space}] When $G$ is a classical simple and $\text{simply}$ connected group, let $\mathcal{M}_{G,\omega}$ be the moduli space of $\omega$-semistable {\em parabolic $G$-bundles} on a complex smooth projective curve of genus $g$. Then
$\mathcal{M}_{G,\omega}$ is of Fano type when $g\geq 3$.\end{thm}

The case $G=\mathrm{SL}_r$ was proved in Theorem 4.12 of \cite{MY}. Our proof mainly depends on the codimension estimates (see Section \ref{subsection codimension}) and the formula of the anti-canonical bundle of the moduli space (see Section \ref{section fano}). We observe that, for the so-called parabolic data $\omega_c$, the moduli space  $\mathcal{M}_{G,\,\omega_c}$ is Fano for a simple and simply connected group $G$ of type $C$ or $D$ (see Proposition \ref{thm type C} $\&$ \ref{thm type D}), and $\omega^{-2}_{\mathcal{M}_{G,\omega_c}}$ is an ample line bundle for the type $B$ case (see Remark \ref{rmk type B}). By analyzing the singularities, we prove that
$\mathcal{M}_{G,\omega_c}$ is of Fano type when $G$ is classical. Then, for general parabolic data $\omega$, our theorem is proved by using the codimension estimates.

As a consequence, the moduli space $\mathcal{M}_{G,\omega}$ is a Mori dream space, and its Cox ring $\mathrm{Cox}(\mathcal{M}_{G,\omega})$ is finitely generated when  $g\geq 3$ (see $\text{Corollary \ref{cor cox ring}}$). Moreover, the same argument shows that, the above conclusions remain valid whenever ${\mathrm{Codim}(\mathrm{Bun}_{G,\omega}\setminus \mathrm{Bun}_{G,\omega}^{rs})\geq 2}$ (even for $g\leq 2$, see $\text{Remark \ref{rmk g=0,1,2}}$). This allows us to prove the second main result of this article, which concerns the {\em algebra of conformal blocks} for genus $g\geq 0$ (see the proof of Theorem \ref{f.g. for singular} for details).

Let $X=(C,\{x\}_{x\in I})\in \overline{\mathcal{M}}_{g,n}$ be an $n$-pointed stable curve of genus $g$, where $n=|I|$. Let
$\mathfrak{g}$ be a simple finite dimensional Lie algebra,
$\ell \in \mathbb{Z}_{\geq 0}$ be an integer, and $\alpha=\{\alpha_x\}_{x\in I}$ be a collection of dominant integral weights such that $(\alpha_x,\theta)\leq \ell$, where $\theta$ denotes the highest root. 
One can construct a finite dimensional vector space $\mathbb{V}_{X,\mathfrak{g},\ell,\alpha}^\dagger$ (see \cite{TUY} or \cite[Section 3.1]{MR2433154}), the so-called {\em space of conformal blocks}.
Then the direct sum (see (\ref{equation graded})) $$\mathbb{V}_{X,\mathfrak{g}}^{\dag}:=\bigoplus_{{{\ell,\alpha}}}
\mathbb{V}_{X,\mathfrak{g},\ell,\alpha}$$
acquires the structure of a commutative graded $\mathbb{C}$-algebra, called the {\em algebra of conformal blocks}.

\begin{thm}[=Theorem \ref{f.g. for singular}]
For any $n$-pointed stable curve $X =(C,\{x\}_{x\in I})\in \overline{\mathcal{M}}_{g,n}$ and any classical simple Lie algebra $\mathfrak{g}$, the \emph{algebra of conformal blocks} $\mathbb{V}_{X,\mathfrak{g}}^\dagger$ is finitely generated.
\end{thm}

Belkale and Gibney first handled the case $\mathfrak{g}=\mathfrak{sl}_r$ with $n=0$ (see \cite{B.G}).
Then Moon and Yoo proved the case $\mathfrak{g}=\mathfrak{sl}_r$ for arbitrary $n$ (see \cite{MY}) using a completely different method.  The two research groups both asked whether the result holds for any simple Lie algebra $\mathfrak{g}$.
Our proof follows a method similar to that in \cite{MY}. For simple Lie algebras of type $A$ or $C$ with $n=0$ and $g\geq 2$, Wilson also proved that the {\em algebra of conformal blocks} is finitely generated (see \cite{W}).

Let $\Phi: G\rightarrow \mathrm{SL}(V)\subset \mathrm{GL}(V)$ be the standard representation, and set (see (\ref{equation fundamental alcove subset}))
$$\Phi_0^{\ast}:=\{h\in \mathfrak{h}\mid \,\theta_{\mathrm{sl}(V)}(h)< 1, \alpha(h)\geq 0\,\, \forall\, \alpha\in \Delta_+\}\subset \Phi_0^s.$$
We have the following corollary, which generalizes $\text{\cite[Theorem 1.2]{B.G}}$ and \cite[Theorem 1.2]{MY}.
{\begin{cor}[=Corollary \ref{corollary compactify} $\&$ Remark \ref{rmk 5.7}] Let $G$ be a classical simple and simply connected group. Suppose that $g\geq 2$ or $\tau_x=\mathrm{Lie}(a_x)\in \Phi_0^{
\ast}$ for $x\in I$. Then there is a flat family $\mathbb{M}\rightarrow \overline{\mathcal{M}}_{g,n}$ $\text{such that}$
\begin{itemize}
\item[(1)] for $X=(C,\{x\}_{x\in I})\in \mathcal{M}_{g,n}$, we have $\mathbb{M}_X\cong \mathcal{M}_{G,\omega}$$;$
\item[(2)] the fibre $\mathbb{M}_X$ over $X\in \overline{\mathcal{M}}_{g,n}-\mathcal{M}_{g,n}$ is an irreducible normal projective variety.
\end{itemize}
\end{cor}}
For simple Lie algebras of type $A$ or $C$ with $g\geq 2$, Wilson also showed that moduli spaces of semistable principal $G$-bundles on smooth curves admit a limit for $X\in \overline{\mathcal{M}}_{g}-\mathcal{M}_{g}$; this limit is precisely the normalized moduli space of semistable honest singular $G$-bundles on $X$ (see \cite{W}).

We briefly describe the contents of this article. In Section 2, preliminaries about {\em parabolic $G$-bundles} and constructions of the moduli space $\mathcal{M}_{G,\omega}$ of semistable {\em parabolic $G$-bundles} are recalled,  codimension estimates related to the loci of stable {\em parabolic $G$-bundles} and regularly stable {\em parabolic $G$-bundles} in the stack $\mathrm{Bun}_{G,\omega}$ of {\em parabolic $G$-bundles} are given, and the {\em algebra of conformal blocks} is introduced. In $\text{Section 3}$, we give the formula of the anti-canonical bundle of the moduli space $\mathcal{M}_{G,\,\omega_c}$, for the so-called canonical parabolic $\text{data $\omega_c$}$. In Section 4, we prove that the moduli spaces are of Fano type. $\text{Section 5}$ is devoted to showing the finite generation of the \emph{algebra of conformal blocks} for a classical simple Lie algebra.

\medskip

{\it Acknowledgements:} The authors would like to thank J. Heinloth for helpful discussions. The authors are also grateful to the anonymous $\text{referees}$ for several helpful suggestions, correcting our $\text{misunderstanding}$ of the reference \cite{B.G}, and pointing out the missing reference \cite{W}.

\section{Preliminaries}

\subsection{Moduli spaces of parabolic $G$-bundles}\label{Section equivariant}

Let $X$ be a scheme, $G$ be an affine algebraic group, and $E$ be a principal $G$-bundle on $X$.
Given a representation $\rho:G\to {\rm GL}(V)$, we use $E_{{\rm GL}(V)}$ to denote the associated ${\rm GL}(V)$-bundle, and $E(V)$ to denote the associated vector bundle
on $X$. But, when ${\rm dim}(V)=1$ (i.e., ${\rho:G\to \mathbb{G}_m}$ is a character of $G$), we use $E(\rho)$ to denote the associated line bundle $\text{on $X$}$.

Let $C$ be a smooth projective curve over $\mathbb{C}$, and $I\subset C$ be a finite subset. Fix a parabolic subgroup ${P_x\subset G}$ for each $x\in I$. Let $X^*(P_x):=\mathrm{Hom}(P_x,\mathbb{G}_m)$ be the character group, $X^*(P_x)^\vee_{\mathbb{Q}}:=\mathrm{Hom}(X^*(P_x),\mathbb{Q})$ be the rational cocharacter group of $P_x$, and
\[
X^*(P_x)^\vee_{\mathbb{Q},+}:=\left\lbrace {a}_x \in X^*(P_x)^\vee_{\mathbb{Q}}\bigg|
\begin{matrix} \aligned
&\, \text{for the given subgroup}\,  P_x\subsetneq P, \\&  \,  \text{we have}\,\, {a}_x(\mathrm{det}_{P} \otimes \mathrm{det}^{-1}_{P_x})<0 \endaligned
\end{matrix}
\right\rbrace
.\]

\begin{defn}\label{defn 1.1} (1) By a family of {\em quasi-parabolic $G$-bundles} of type $\{P_x\}_{x\in I}$ over a scheme $T$, we mean a $G$-bundle $\sE_T$ on $C\times T$ with
a set of sections $\{S_{x,T}:\{x\}\times T\to \sE_T|_{\{x\}\times T}/P_x\}_{x\in I}$. In particular, we obtain a {\em quasi-parabolic $G$-bundle} on $C$ when $T$ is a point.

(2) By a {\em parabolic $G$-bundle}, we mean a {\em quasi-parabolic $G$-bundle} $E$ of type $\{P_x\}_{x\in I}$ on $C$, together with the given weight
 \[{a}=\{{a}_x\}_{x\in I}, \quad {a}_x\in X^{\ast}(P_x)^{\vee}_{\mathbb{Q},+}.\]
$\omega:=(\{P_x\}_{x\in I}, {a})$ is called the parabolic data.
\end{defn}

For any parabolic subgroup $P\subset G$, $G\to G/P$ is a $P$-bundle, and
\begin{equation}\label{2.1} X^*(P)\xrightarrow{L} {\rm Pic}(G/P),\quad \chi \mapsto G(\chi)^{-1}:=L(\chi)
\end{equation}
is an isomorphism of groups when $G$ is simple and simply connected.

To see that ${a}_x$ indeed determines an ample line bundle $L(k\chi_{a_x})$ (for some $k\in \mathbb{Z}^+$) on $G/P_x$ ($x\in I$), let $X_*(P_x)$ be the group of one-parameter subgroups $\mathbb{G}_m\to P_x$
($1$-PS for short). Then any $\lambda\in X_*(P_x)$ determines uniquely a parabolic subgroup $P(\lambda)$ which is generated by the maximal torus $T$ (containing $\lambda$) and
 root groups $U_{{a}}$ with ${a}(\lambda)\ge 0$ (see \cite[Remark 1.4]{RR}).
Since $P(\lambda)=P(n\lambda)$ for $n\in \mathbb{Z}^+$, $P(\lambda)$ is well-defined for $\lambda\in X_*(P_x)_{\mathbb{Q}}$, and it depends only on the conjugacy class of $\lambda$. Note that
$$X^*(P_x)\times X_*(P_x)\to \mathbb{Z},\quad (\chi,\lambda)\mapsto \chi(\lambda)$$
is invariant under the conjugation. Therefore, the conjugacy class of $\lambda$ determines a rational cocharacter $[\lambda]\in X^*(P_x)^\vee_{\mathbb{Q}}$. In particular,
$$[\lambda]\in X^*(P_x)^\vee_{\mathbb{Q},+} \,\,\Leftrightarrow\,\,P_x=P(\lambda).$$
We will use the same symbol $\lambda$ to denote both rational $1$-PS $\lambda$ and rational cocharacter $[\lambda]$. Then $P_x=P({a}_x)$ holds for the rational $1$-PS ${a}_x$, which determines a dominant rational character $\chi_{{a}_x}\in X^*(P_x)_{\mathbb{Q}}$. In fact, let $T\subset P_x$ and $(-,\,-)_{\mathfrak{g}}:X_*(T)\times X_*(T)\to \mathbb{Z}$ be the pairing defined by the normalized Killing form. For $\lambda\in X_*(T)_{\mathbb{Q} }$,  $\chi_{\lambda}\in X^*(T)_{\mathbb{Q}}$ is defined by
\begin{equation}\label{2.2}
\chi_{\lambda}(\lambda'):=(\lambda,\,\lambda')_{\mathfrak{g}},\quad \forall\,\,\lambda'\in X_*(T)_{\mathbb{Q}}.
\end{equation}
When $P_x=P({a}_x)$, one has $\chi_{{a}_x}\in X^*(P_x)_{\mathbb{Q}}$ (see \cite[Remark 1.11 (c)]{RR}), which determines an ample line bundle $L(k\chi_{a_x})$ on $G/P_x$ for some $k\in \mathbb{Z}^+$. Moreover, we have

\begin{prop}[see Proposition 5.1.1 of \cite{Heinloth}]\label{prop2.3} The above construction $$\lambda\mapsto (P(\lambda),\,\chi_{\lambda})$$ is a surjective map from the set of $1$-PS of $G$ to the set of
pairs $(P,\,\chi)$, where  $\chi$ is a dominant character of the parabolic subgroup $P\subset G$.
\end{prop}

Let ${\rm Bun}_G$ be the stack of $G$-bundles on $C$, and ${\rm Bun}_{G,\,\{P_x\}_{x\in I}}\xrightarrow{\pi} {\rm Bun}_G$
be the stack of \emph{quasi-parabolic $G$-bundles} of type $\{P_x\}_{x\in I}$. When $G$ is simple and simply connected, it is well-known (see \cite{L.S}) that
\begin{equation}\label{4.1} \mathrm{Pic}(\mathrm{Bun}_{G,\{P_x\}_{x \in I}}) \xrightarrow{\sim} \mathbb{Z} \times \prod_{x \in I} X^*(P_x) \cong \mathbb{Z} \times \prod_{x \in I} \mathrm{Pic}(G/P_x).
\end{equation}

Let $\sE$ be the universal $G$-bundle on $C\times {\rm Bun}_{G,\,\{P_x\}_{x\in I}}$, and
\begin{equation}\label{4.2}
\{\,{\rm Bun}_{G,\,\{P_x\}_{x\in I}}\xrightarrow{S_x} \sE_x/P_x\,\}_{x\in I},\,\,\sE_x:=\sE|_{\{x\}\times {\rm Bun}_{G,\,\{P_x\}_{x\in I}}}
\end{equation}
be the universal sections. For the adjoint representation ${{\rm Ad}: G\to {\rm SL}(\mathfrak{g})}$,
let $\sE(\mathfrak{g})$ be the associated vector bundle on $C\times {\rm Bun}_{G,\,\{P_x\}_{x\in I}}$, and
\begin{equation}\label{4.3}
\sD_{{\rm Ad}}:={\rm det}{\rm R}pr_{2_*}\sE(\mathfrak{g})^{-1},
\end{equation}
where ${\rm det}{\rm R}pr_{2_*}\sE(\mathfrak{g})$ is the determinant line bundle on ${\rm Bun}_{G,\,\{P_x\}_{x\in I}}$, and $pr_2$ is the projection
$C\times \mathrm{Bun}_{G,\{P_x\}_{x\in I}}\rightarrow \mathrm{Bun}_{G,\{P_x\}_{x\in I}}$. Let
$\sE(\chi_x)$ be the line bundle associated to $P_x$-bundle $\sE_x\to \sE_x/P_x$, and
\begin{equation}\label{4.4}
\sL(\chi_x):=S_x^*\sE(\chi_x)^{-1},\,\,\,\forall\,\, \chi_x\in X^*(P_x).
\end{equation}
We collect some basic facts for further applications.

\begin{lem}\label{thm forgetting map of stack}\label{lem anti-canonical}
Let $P_x\subset G$ $({x \in I\cup J})$ be parabolic subgroups. Assume that $I\cap J=\emptyset$.
\begin{enumerate}
\item
The forgetful morphism $\mathrm{Bun}_{G,\{P_x\}_{x \in I\cup J}} \to \mathrm{Bun}_{G,\{P_x\}_{x \in I}}$
is a locally trivial bundle with fibres $\prod_{x \in J} G/P_x.$

\item
Let $Q_x \subset P_x$ $(x\in I)$ be parabolic subgroups of $G$. Then there exists a natural morphism $\mathrm{Bun}_{G,\{Q_x\}_{x \in I}} \to \mathrm{Bun}_{G,\{P_x\}_{x \in I}}$, which is a locally trivial bundle with fibres $\prod_{x\in I}P_x/Q_x.$

\item Assume that $B_x\subset G$ $(x\in I)$ are Borel subgroups, and $\rho$ is the half sum of all positive roots. We have
\begin{equation}\label{4.5}
\omega^{-1}_{{\rm Bun}_{G,\,\{B_x\}_{x\in I}}}=\sD_{{\rm Ad}}\otimes\bigotimes_{x\in I}S_x^*\sE(2\rho)^{-1}:=\sD_{{\rm Ad}}\otimes\bigotimes_{x\in I}\sL(2\rho).
\end{equation}
\end{enumerate}
\end{lem}

Let $E:=(E,\{s_x\}_{x\in I})$ be a \emph{quasi-parabolic $G$-bundle} of type $\{P_x\}_{x \in I}$ on $C$. Given a parabolic subgroup $Q \subset G$, a character $\chi$ of $Q$, and a reduction $E_Q$ of $E$ to $Q$, via a trivialization $$Q_x:={\rm Aut}_Q(E_Q|_{x}) \cong Q,$$
we obtain a character $\chi_x$ of $Q_x$ for $x \in I$.  Similarly, $\{x\}\xrightarrow{s_x} E|_x/P_x$ defines a
$P_x$-bundle $s_x^*E|_x$, and $P_{s_x}:={\rm Aut}_{P_x}(s_x^*E|_x)\subset {\rm Aut}_G(E|_x)=G$ is canonically isomorphic to $P_x$ up to an inner automorphism of $P_x$. Thus, there is a canonical isomorphism $X^{\ast}(P_x)^{\vee}_{\mathbb{Q},+}\cong X^{\ast}(P_{s_x})^{\vee}_{\mathbb{Q},+}$, and we take
${a}_{s_x}\in X^{\ast}(P_{s_x})^{\vee}_{\mathbb{Q},+}$ to be the image of ${a}_x\in X^{\ast}(P_x)^{\vee}_{\mathbb{Q},+}$.
Clearly, the two parabolic subgroups $P_{s_x}$ and $Q_x$ of $G$ share a maximal torus $T_x$, and ${a}_{s_x}$ can be viewed as a rational $1$-PS of $T_x$. Hence, for any character $\chi_x$ of $Q_x$, the value $\langle \chi_x,{a}_{s_x} \rangle$ is well-defined.

\begin{defn}\label{definition semistability}
(1) Let $Q \subset G$ be a parabolic subgroup, and $E_Q$ be a reduction of $E$ to $Q$. For any nontrivial dominant character $\chi$ of $Q$, let $E_Q(\chi)$ be the associated line bundle on $C$, and
\[
{a}\text{-}\mathrm{deg}(E_Q(\chi)):=\deg(E_Q(\chi))+\sum_{x \in I} \langle \chi_x,{a}_{s_x} \rangle.
\]

(2) A \emph{parabolic $G$-bundle} $E$ of type $\{P_x\}_{x \in I}$ on $C$ is called  $\omega$-stable (resp.  $\omega$-semistable) if, for any parabolic subgroup $Q \subset G$, any $\text{reduction}$ $E_Q$ of $E$ to $Q$, and any nontrivial dominant character $\chi$ of $Q$,
\[
{a}\text{-}\mathrm{deg}(E_Q(\chi)) < 0\,\, (\text{resp}.\leq)
\] holds.
If there exist only $\omega$-stable \emph{parabolic $G$-bundles}, the parabolic data $\omega$ is called generic.
\end{defn}

Let $\theta$ be the highest root of $\mathfrak{g}={\rm Lie}(G)$, and $\mathfrak{h}\subset \mathfrak{g}$ be a Cartan subalgebra.
Set
\begin{equation}\label{equation fundamental alcove}\Phi_0:=\{h\in \mathfrak{h}\mid \,\theta(h)\leq 1, \alpha(h)\geq 0\,\, \forall\, \alpha\in \Delta_+\}\end{equation} (the fundamental alcove), and $\Phi^s_0\subset \Phi_0$ is defined to be the subset with $\theta(h)<1$,
where $\Delta_+$ consists of all simple positive roots.
\begin{rmk}\label{rmk nonzero} Let $\omega=(\{P_x\}_{x\in I}, {a})$ be the parabolic data. Then for any $x\in I$, $a_x\in X^{\ast}(P_x)^{\vee}_{\mathbb{Q},+}$ (see Definition \ref{defn 1.1}) and $\mathrm{Lie}(a_x)\in \mathfrak{h}$
must be nonzero.
\end{rmk}
\goodbreak
\begin{ex}\label{ex 2.7} When $G=\mathrm{SL}(V)$, ${a}'_x=d_x\cdot {a}_x:\mathbb{G}_m\to \mathrm{SL}(V)$ defines
$$V=\sum^{l_x+1}_{i=1}V^i, \quad V^i=\{\,v\in V\,|\,{a}'_x(z)(v)=z^{c_i}v,\,\,\forall\,\,z\in \mathbb{G}_m(\mathbb{C})=\mathbb{C}^*\},$$
where $c_1<c_2<\cdots<c_{l_x+1}$. Let $Q_i:=V^{l_x-i+2}\oplus\cdots\oplus V^{l_x+1}$ and
$$Q_{\bullet}({a}_x): V\twoheadrightarrow Q_{l_x}\twoheadrightarrow\cdots \twoheadrightarrow Q_i\twoheadrightarrow Q_{i-1}\twoheadrightarrow\cdots \twoheadrightarrow Q_1\to 0,$$ where $Q_i\twoheadrightarrow Q_{i-1}$ is the projection (its kernel is $V^{l_x-i+2}$). Let
$$ \mathrm{Flag}_{\vec r(x)}(V),\quad \vec{r}(x):=(r_1(x),r_2(x),...,r_{l_x}(x)),\quad r_i(x)={\rm dim}(Q_i)$$
be the variety of quotient flags of type $\vec{r}(x)$. Clearly, $Q_{\bullet}({a}_x)\in \mathrm{Flag}_{\vec r(x)}(V)$, and $P({a}_x)\subset \mathrm{SL}(V)$ is the stabilizer
of $Q_{\bullet}({a}_x)$. If ${a}_x\in X^*(P_x)^\vee_{\mathbb{Q},+}$, then $P_x=P({a}_x)$. In this case, set $b_i(x):=c_{i+1}-c_i$, and 
\begin{equation}\label{2.3}\chi_{{a}_x}=\frac{b_1(x)}{d_x}\omega_{r_1(x)}+\cdots +\frac{b_{l_x}(x)}{d_x}\omega_{r_{l_x}(x)},\end{equation}
where $\omega_i$ ($1\le i\le r-1$) are fundamental weights of ${\rm SL}(V)$. Let $k$ be a positive integer such that $d_i(x):=kb_i(x)/d_x$ ($1\le i\le l_x$) are
integers, and $a_1(x)=0$, $a_i(x)=d_1(x)+\cdots +d_{i-1}(x)$ for $2\leq i\leq l_x+1$. Then
\begin{equation}\label{2.4} L(k\chi_{{a}_x})=\bigotimes^{l_x}_{i=1}{\rm det}(\sQ_{x,i})^{d_i(x)}\,\,\,\text{on}\,\,\,\mathrm{SL}(V)/P_x=\mathrm{Flag}_{\vec r(x)}(V),\end{equation}
where $V\otimes\sO_{\mathrm{SL}(V)/P_x}\twoheadrightarrow \sQ_{x,l_x}\twoheadrightarrow\cdots \twoheadrightarrow \sQ_{x,1}\twoheadrightarrow 0$ is the universal flag. Moreover,  $\mathrm{Lie}({a}_x)\in \Phi^s_0$ if and only if $a_{l_x+1}(x)=d_1(x)+\cdots +d_{l_x}(x)< k$.
\end{ex}

\begin{thm}[see \cite{Kumar}]\label{thm existence of moduli space}
For a simple and simply connected algebraic group $G$ and $\omega=(\{P_x\}_{x\in I},{a})$ with $\mathrm{Lie}(a_x)\in \Phi^s_0$, there exists a coarse moduli space $\mathcal{M}_{G,\omega}$ of $S$-equivalent classes of $\omega$-semistable \emph{parabolic $G$-bundles} of type $\{P_x\}_{x \in I}$ with a fixed topological type on $C$. Moreover, $\mathcal{M}_{G,\omega}$ is a normal projective variety with rational singularities.
\end{thm}
The construction of the moduli space of semistable \emph{parabolic $G$-bundles} first appeared in \cite{BR}, and its
generalization to arbitrary characteristic was given in \cite{Heinloth}. Both constructions require the admissibility condition on the parabolic weight (see \cite[\S 4.1]{Heinloth}).
Here, following the method of Balaji and Seshadri (see \cite{B.C}) that uses the equivariant $G$-bundles on a cover of $C$, we sketch a construction of the moduli space for all ${a}$ with $\mathrm{Lie}({a}_x)\in \Phi^s_0$.

Let $\omega=(\{P_x\}_{x\in I},{a})$ be the given parabolic data, and $\tau_x=\mathrm{Lie}(a_x)\in \Phi_0^s$ be the rational element, i.e., $\tau_x=\frac{\bar \tau_x}{d_x}$ for some integer $d_x\geq 2$ and $\mathrm{Exp}(2\pi i\bar \tau_x)=1$. For a smooth $|I|$-pointed curve $C$ with $|I|\geq 1$ (and, when $C=\mathbb{P}^1$, assuming $|I|\geq 3$), by a result of Selberg (see \cite{Selberg}), there exists a smooth irreducible projective curve $\widehat C$ and a Galois cover $\pi:\widehat C\rightarrow C$ with a finite Galois group $A$, where $A$ acts freely on $\widehat{C}\setminus \pi^{-1}(\{x\}_{x\in I})$, and the isotropy subgroup $A_{\hat x}$ (for any $\hat x\in \pi^{-1}(x)$) is cyclic of order $d_x$.
An $A$-equivariant $G$-bundle (or an $(A,G)$-bundle for short) $\widehat{E}$ on $\widehat{C}$ is defined to be a principal $G$-bundle together with an action of $A$ on its total space by bundle automorphisms preserving the action of $G$, and it is called $A$-semistable (resp. $A$-stable) if (2) of Definition \ref{definition semistability} is satisfied for any parabolic subgroup $Q\subset G$ and any $A$-equivariant reduction $\widehat E_Q$ of $\widehat E$ to $Q$. Let  $\mathrm{Bun}_{G}^{A,\vec{\tau}}(\widehat{C})$ be the stack of $A$-equivariant $G$-bundles of local type $\vec \tau=\{\bar \tau_x\}_{x\in I}$ on $\widehat{C}$. Then it suffices to show the existence of the moduli space of $A$-semistable $G$-bundles on $\widehat{C}$ by the following result.

\begin{thm}[see Theorem 6.1.15 and Theorem 6.1.17 of \cite{Kumar}]\label{correspondence} For $\tau_x\in \Phi^s_0$ $(x\in I)$, there exists an isomorphism
\begin{equation}\label{equation isomorphism stacks}
\mathrm{Bun}_{G}^{A,\vec{\tau}}(\widehat{C})\cong \mathrm{Bun}_{G,\{P_x\}_{x\in I}}.\end{equation}
Under the isomorphism,  $A$-semistable $($resp. $A$-stable$)$ $G$-bundles on  $\widehat{C}$ correspond to $\omega$-semistable $($resp. $\omega$-stable$)$ {\em parabolic $G$-bundles} on $C$.
\end{thm}

To construct the moduli space of $A$-semistable $G$-bundles on $\widehat{C}$, we fix an embedding $i:G\hookrightarrow \mathrm{SL}(V)\subset\mathrm{GL}(V)$ with $r:=\mathrm{dim}(V)$. For an $A$-semistable $G$-bundle $\widehat{E}$ on $\widehat{C}$, the vector bundle $\widehat E(V)$ is also semistable with trivial determinant.
Fix an $A$-stable finite subset $\{y_1,\ldots, y_b\}\subset \widehat{C}$ of points, let $\vec y=\vec y(d'):=d'\sum_{j=1}^b y_j$ and $\sO_{\widehat{C}}(\vec y)$ be the
corresponding $A$-equivariant line bundle. Choose $d'\gg 0$. Then, for any semistable vector bundle $E$ of rank $r$ and degree $0$ on $\widehat{C}$, we have
\begin{enumerate}
\item $H^1(\widehat{C},E(\vec y))=0$, and

\item $E(\vec y)$ is generated by its global sections.
\end{enumerate}
Let $P(z)=r(1-\hat{g})+rd'bz$, $\sO_{\widehat{C}}(-1)=\sO_{\widehat{C}}(\vec y)^{-1}$, and $\bold{Q}$ be the Quot scheme of quotients $(\sO_{\widehat{C}}\otimes \mathbb{C}^{\oplus P(N)})\otimes \sO_{\widehat{C}}(-N)\rightarrow \hat{F}\rightarrow 0$ of rank $r$ and degree $0$ on $\widehat{C}$. Then there is on $\widehat{C}\times \bold{Q}$ a universal quotient $$(\sO_{\widehat{C}\times \bold{Q}}\otimes \mathbb{C}^{\oplus P(N)})\otimes \sO_{\widehat{C}\times \bold{Q}}(-N)\rightarrow \hat{\mathcal{F}}\rightarrow 0.$$

Fix a representation $\hat{\tau}$ of $A$ on $\mathbb{C}^{\oplus P(N)}$, then the action of $A$ on $\widehat{C}$ induces an action of $A$ on $\sO_{\widehat{C}}\otimes \mathbb{C}^{\oplus P(N)}$. Hence $\sO_{\widehat{C}\times \bold{Q}}\otimes \mathbb{C}^{\oplus P(N)}$ is an $A$-equivariant vector bundle on $\widehat{C}\times \bold{Q}$, which induces a canonical action of $A$ on $\bold{Q}$ such that $\hat{\mathcal{F}}$ is an $A$-equivariant coherent sheaf. Let $\bold{Q}^A\subset \bold{Q}$ be the closed subscheme of $A$-invariant points, and $\mathcal{R}_{\hat{\tau}}^{ss}\subset \bold{Q}^A$ (resp. $\mathcal{R}_{\hat{\tau}}^{s}\subset \bold{Q}^A$) be the open subset consisting of $A$-semistable (resp. $A$-stable) vector bundles. Let $\mathcal{R}^{ss}(G)\to\mathcal{R}^{ss}_{\hat{\tau}}$ be the $\mathcal{R}^{ss}_{\hat{\tau}}$-scheme representing the following functor
 $${\small \Gamma(i,\hat{\mathcal{F}})(f:T\rightarrow \mathcal{R}^{ss}_{\hat{\tau}}):=\left\{\aligned &\text{the set of $A$-equivariant sections $\sigma$ of $\hat{\mathcal{F}}_f/G$}\\& \text{where $\hat{\mathcal{F}}_f:=(id_{\widehat{C}}\times f)^{\ast}(\hat{\mathcal{F}})$ on $\widehat{C}\times T$} \endaligned\right\},}$$
 and $\mathcal{R}^{ss}_{\tau}(G)\subset \mathcal{R}^{ss}(G)$ be the closed subset consisting of $G$-bundles of topological type $\tau$ (and of local type $\vec \tau$), where the action of $\mathfrak{G}:=\mathrm{GL}^A_{P(N)}$ ($A$-invariants under the conjugation action of $A$ on $\mathrm{GL}_{P(N)}$ induced from $\hat{\tau}$) on $\mathcal{R}^{ss}_{\hat{\tau}}$ canonically lifts to it. Then the $\mathfrak{G}$-invariant morphism $f_{\tau}:\mathcal{R}^{ss}_{\tau}(G)\rightarrow \mathcal{R}^{ss}_{\hat{\tau}}$
is affine, and the quotient
\begin{equation}\label{equation affine}
\mathcal{M}_{G}^{A,\vec{\tau}}(\widehat{C}):=\mathcal{R}^{ss}_{\tau}(G)//\mathfrak{G}\,\,(\rightarrow \mathcal{R}^{ss}_{\hat{\tau}}//\mathfrak{G}:=\mathcal{M}_{\mathrm{GL}(V)}^{A}(\hat C) )\end{equation}
is the moduli space of $A$-semistable $G$-bundles of fixed topological type $\tau$ (and local type $\vec \tau$) on $\widehat{C}$. Furthermore, the quotient $\mathcal{M}_{G}^{A,\vec{\tau}}(\widehat{C})$ is a normal variety with rational  singularities, since $\text{$\mathcal{R}^{ss}_{\tau}(G)$ is so}$.

Finally, it remains to show the properness of $\mathcal{M}_{G,\omega}\cong\mathcal{M}_{G}^{A,\hat{\tau}}(\widehat{C})$. When $g=g(C)\geq 2$, this follows from the compactness of $\mathcal{M}_{G}^{A,\hat{\tau}}(\widehat{C})$ in the analytic topology (see \cite[Theorem 8.1.7]{B.C}). In general, let $E$ be a \emph{parabolic $G$-bundle} on $C$, and $\widehat{E}$ be its corresponding $A$-equivariant $G$-bundle on $\widehat{C}$ (see Theorem \ref{correspondence}).
Then, for any parabolic subgroup $Q{\subset G}$, there is a bijective correspondence between sections of ${E/Q\rightarrow C}$ and $A$-equivariant sections of ${\widehat{E}/Q\rightarrow \widehat{C}}$, and $\mathrm{deg}(\widehat{E}_Q)=N{a}\text{-}\mathrm{deg}(E_Q)$ (see \cite[\S 6.1 (17)]{Kumar}),
where $N:=|A|$, $\widehat{E}_Q$ and $E_Q$ are the reductions. In the above correspondence, one can show the existence and uniqueness of the canonical reduction of a \emph{parabolic $G$-bundle} on $C$ from an $A$-equivariant $G$-bundle on $\widehat{C}$,
whose canonical filtration is $A$-equivariant. Hence we obtain an analogue of Behrend's theorem (see \cite[Theorem 4.3.2]{Heinloth}), which implies the semistable reduction theorem for $\text{\emph{parabolic $G$-bundles}}$ on $C$ by following  \cite[Theorem 4.4.1]{Heinloth}.

\subsection{Codimension estimates}\label{subsection codimension}
In this section, we prove two codimension estimates (see Proposition \ref{prop codim semistable and stable parabolic} and Proposition \ref{prop codimension regularly stable}) for the loci of unstable (resp. non-regularly stable) \emph{parabolic $G$-bundles} in the stack $\mathrm{Bun}_{G,\omega}$ of \emph{parabolic $G$-bundles}.

We begin with a well-known lemma (see \cite{LM}, for example). To simplify notation, we always assume that $G$ is simple and simply connected.
Let $Q\subset G$ be a parabolic subgroup, and $\mathrm{Bun}^{\geq h}_Q\subset \mathrm{Bun}_G$ be the closed substack consisting of principal $G$-bundles $E$ that admit reductions $E_Q$ of $E$ to $Q\subset G$ with $\mathrm{deg}(E_Q(\mathrm{det}_Q))\geq h$. 
\begin{lem}\label{lem codimension key lem}
The inequality
$$\mathrm{Codim}(\mathrm{Bun}_Q^{\geq h},\mathrm{Bun}_G)\geq \mathrm{dim}R_u( Q)(g-1)+h$$ holds,
where
$R_u(Q)$ is the unipotent radical of $Q$.
\end{lem}

Let $Q\subset G$ be a standard parabolic subgroup, with respect to the given Borel subgroup $B\subset G$. For each $x\in I$, set a proper closed subset $$Z_{Q,\,x}:=\{s_x\in G/P_x\mid \, \langle\mathrm{det}_Q,{a}_{s_x}\rangle\geq 0\}\subsetneq G/P_x.$$
By a formula of $\langle\mathrm{det}_Q,{a}_{s_x}\rangle$ (see formula (3) on page 454 of \cite{Heinloth}) and a suitable choice of the parabolic weight ${a}=\{{a}_x\}_{x\in I}$, we may assume

\begin{itemize} \item [(1)] there exists a positive number $n_1$ such that for any $s_x\in G/P_x$ and any $Q$, we have $\langle \mathrm{det}_Q, {a}_{s_x}\rangle< {a}_{s_x}\rangle< n_1$;
\item [(2)] there exists a positive number $n_2$ such that if $\langle \mathrm{det}_Q, {a}_{s_x}\rangle<0$, then $\langle \mathrm{det}_Q, {a}_{s_x}\rangle<-n_2$
for any $s_x\in G/P_x$ and any $Q$.
\end{itemize}
Let $h$ be an integer such that $$\min_{Q\subset G}\{\,\mathrm{dim}R_u(Q)(g-1)+h\,\}=2.$$
The following result and its proof are variants of \cite[Proposition 4.2.2]{Heinloth}.

\begin{prop}\label{prop codim semistable and stable parabolic}
Assume that the weight ${a}$ satisfies assumptions $(1)$ and $(2)$. Then, if $|I|\geq \frac{2(n_1+n_2)+h}{n_2}$, we have
$$\mathrm{Codim}(\mathrm{Bun}_{G,\omega}\setminus \mathrm{Bun}^s_{G,\omega})\geq 2.$$
\end{prop}

\begin{proof} For a standard $Q\subset G$, let $\mathrm{Bun}^{<h}_Q:=\mathrm{Bun}_G\setminus \mathrm{Bun}^{\geq h}_Q$,
$$\mathrm{Bun}^{<h}_G:=\bigcup_{Q\subset G}\mathrm{Bun}^{<h}_Q, \text{and} \,\, \mathrm{Bun}_{G,\omega}^{< h}:=\pi^{-1}(\mathrm{Bun}^{<h}_G)\subset \mathrm{Bun}_{G,\omega},$$
where $\pi: \mathrm{Bun}_{G,\omega}\to \mathrm{Bun}_{G}$. Then, by Lemma \ref{lem codimension key lem}, it suffices to prove
$$\mathrm{Codim}(\mathrm{Bun}_{G,\omega}^{< h}\setminus \mathrm{Bun}^s_{G,\omega})\geq 2.$$

Let $(E,\{s_x\}_{x\in I})\in \mathrm{Bun}_{G,\omega}^{< h}$.  If
${a}\text{-}\mathrm{deg}(E_Q(\mathrm{det}_Q))\geq 0$
for some reduction $E_Q$ of $E$ to $Q\subset G$, the subset
$$J=\{x\in  I\mid s_x\in Z_{Q,x}\subset G/P_x\,\}$$
has at least two points. Indeed, when $|J|<2$,  ${a}\text{-}\mathrm{deg}(E_Q(\mathrm{det}_Q))$
$$\aligned =\mathrm{deg}(E&_Q(\mathrm{det}_Q))+\sum_{x\in J}\langle \mathrm{det}_Q,{a}_{s_x}\rangle+\sum_{x\in I\setminus J}\langle\mathrm{det}_Q,{a}_{s_x}\rangle \\ < & h+n_1\cdot |J|-n_2\cdot (|I|-|J|)<0.\endaligned$$
Thus, the fibers of $X_Q:=$
$$\{\,(E,\{s_x\}_{x\in I})\in \mathrm{Bun}_{G,\omega}^{\leq h}\mid {a}\text{-}\mathrm{deg}(E_Q(\mathrm{det}_Q))\geq 0\,\}\to \mathrm{Bun}^{<h}_G$$
are contained in the closed subset
$$\prod_{x\in J}Z_{Q,x}\times \prod_{x\in I\setminus J} G/P_x\subset \prod_{x\in I}G/P_x,$$
which has codimension at least $|J|\ge 2$. Finally, we are done by running through all standard parabolic subgroups of $G$.
\end{proof}
\goodbreak

\begin{rmk}\label{rmk codim stable principal} Although the proof shows that we may make the codimension
as large as we wish by taking $|I|\gg 0$ (as pointed out in $\text{\cite[Remark 4.2.3]{Heinloth}}$), we still
expect a lower bound of $$\mathrm{Codim}(\mathrm{Bun}_{G,\omega}\setminus \mathrm{Bun}^s_{G,\omega})$$ in terms of $|I|$ and the genus $g$ of $C$ (as in the case $G={\rm GL}_r,\, {\rm SL}_r$, see \cite{Su1},\, \cite{Su3}). Note that there already exist some lower bounds in terms of $g$, which are complementary to the above estimates:

(1) By \cite[Theorem II.6]{F} and the result in Section 3.1 of \cite{LM}, we have
$$\mathrm{Codim}(\mathrm{Bun}_G\setminus \mathrm{Bun}_G^{s}) \geq 2$$ whenever $g \geq 2$, except for $g=2$ and $G=\mathrm{SL}_2$.

(2) When $g$ is large enough, one can expect the required codimension estimates for arbitrary parabolic data $\omega$. For example, we have
$$\aligned &
\mathrm{Codim}(\mathrm{Bun}_{G,\omega}\setminus \mathrm{Bun}_{G,\omega}^{ss})\geq 2 \,\,\text{when}\,g\geq 2; \\&
\mathrm{Codim}(\mathrm{Bun}_{G,\omega}\setminus \mathrm{Bun}_{G,\omega}^{s})\geq 2
\,\, \text{when}\, g\geq 3 .\endaligned$$
We take \cite[Proposition 6.3.1]{PP} as a reference, where a more general case, namely Bruhat-Tits group scheme torsors, is considered.\end{rmk}

Recall that a stable \emph{parabolic $G$-bundle $E$} is called regularly stable, if
$Z(G) \to \mathrm{Aut}(E)$
(automorphisms as a \emph{parabolic $G$-bundle}) is an isomorphism.
For the {principal $G$-bundle}, the closed subset of $\mathrm{Bun}_G^{s}$ parametrizing stable bundles that are not regularly stable has codimension at least two 
when $g\geq 2$, except for $g=2$ and $G=\mathrm{SL}_2$ (see \cite[Theorem II.6]{F} and \cite[Proposition 11.6]{L} for details).
For further applications, we prove its parabolic analogue under certain conditions.

\begin{prop}\label{prop codimension regularly stable} Let $\omega=(\{B_x\}_{x\in I},{a})$ where  $B_x\subset G$ $($$x\in I$$)$ are Borel subgroups. Assume that there exists an $\omega$-stable {\em parabolic $G$-bundle}.
Then we have
$$
\mathrm{Codim}(\mathrm{Bun}^s_{G,\omega}\setminus \mathrm{Bun}^{rs}_{G,\omega})\geq 2g-3+|I|,
$$ where $\mathrm{Bun}_{G,\omega}^{rs}\subset \mathrm{Bun}_{G,\omega}$ denotes the open substack consisting of regularly stable {\em parabolic $G$-bundles}.
\end{prop}
\begin{proof}
Our proof follows a similar argument in \cite[\S 11]{L}. In fact, we will prove that for any
$(E,\{s_x\}_{x\in I})\in \mathrm{Bun}^s_{G,\omega}\setminus \mathrm{Bun}^{rs}_{G,\omega}$,
$E$ must lie in a fixed closed substack of codimension at least $2g-3$, and $s_x$ must lie in a fixed proper closed subset
of $G/B_x$ for each $x\in I$. 

First, let $\sigma\in {\rm Aut}(E, \{s_x\}_{x\in I})\setminus Z(G)$, and $g_{\sigma}\in G$ such that
$\sigma(e)=e\cdot g_{\sigma}$ for ${e\in E}$.
Then $g_{\sigma}$ has finite order, and its reductive centralizer ${L:=C(g_{\sigma})\subsetneq G}$. By \cite[Lemma 11.1]{L}, $E$ has an $L$-structure, and it lies in the image of the natural morphism $\mathrm{Bun}_{L}\rightarrow \mathrm{Bun}_G$,
with $$\mathrm{dim}\,\mathrm{Bun}_G-\mathrm{dim}\,\mathrm{Bun}_L\geq (2g-3)\mathrm{dim}\,Z(L)\geq 2g-3.$$

To continue, $s_x\in E|_x/B_x\cong G/B_x$ is $\sigma$-invariant, and it lies in the locus of fixed points of $G/B_x=\{\,\text{all Borel subgroups $B\subset G$}\,\}$
under the action of $g_{\sigma}$ by the conjugation $B\mapsto g_{\sigma}B g_{\sigma}^{-1}$. Since
$$g_{\sigma}\notin Z(G)=\bigcap_{B\subset G}B,$$
there is one Borel subgroup $B$ such that $g_{\sigma}\notin B=N_G(B)$ (i.e., $g_{\sigma}B g_{\sigma}^{-1}\neq B$). Hence the locus of fixed points (under the action of $g_{\sigma}$) is a proper closed subvariety of $G/B_x$, and the non-regularly stable locus of $\mathrm{Bun}^s_{G,\omega}$ has codimension at least $2g-3+|I|$. We are done.
\end{proof}

\begin{rmk}\label{codim regularly stable principal}
When $g\geq 2$, we have $$\mathrm{Codim}(\mathrm{Bun}^s_{G,\omega}\setminus \mathrm{Bun}^{rs}_{G,\omega}) \geq 2$$ for arbitrary $\{P_x\}_{x\in I}$, except for $g=2$, $G=\mathrm{SL}_2$, and $|I|=0$.
\end{rmk}
\subsection{Algebra of conformal blocks} 
Now, we briefly outline the construction of the {\em algebra of conformal blocks} (see \cite[\S~3.1]{MR2433154} for details). 

First, we construct the space of conformal blocks using the representation theory. Fix a simple Lie algebra $\mathfrak{g}$.
For each dominant integra weight $\lambda$, let $V_{\lambda}$ denote the corresponding irreducible $\mathfrak{g}$-representation. Fix an integer $\ell\geq 0$ satisfying $(\theta, \lambda)\leq \ell$. Then there is an integrable highest weight module $H_{\ell, \lambda}\supset V_{\lambda}$ of affine Kac-Moody algebra $\hat{\mathfrak{g}}$ of level $\ell$ associated to $\mathfrak{g}$. 
Here, $\theta$ is the highest root, and $(-,-)$ is the normalized Killing form.
For a non-negative integer $\ell$ and dominant integral weights $\alpha=\{\alpha_x\}_{x\in I}$ with $(\theta, \alpha_x)\leq \ell$, set $H_{\ell,\alpha}:=\otimes_{x\in I} H_{\ell, \alpha_x}$. Fix $X = (C, \{x\}_{x\in I})\in \overline{\mathcal{M}}_{g,n}$. Define
\[
H^0(C, \mathcal{O}_C(\ast \sum_{x\in I}x):=
\varinjlim\limits_{m} H^0(C, \mathcal{O}_C(m\sum_{x\in I}x))
\]
and $\hat{\mathfrak{g}}(X):=\mathfrak{g}\otimes H^0(C, \mathcal{O}_C(\ast\sum_{x\in I}x))$. This space has a natural Lie algebra structure, and it acts on $H_{\ell, \alpha}$. Define 
$
\mathbb{V}_{X,\mathfrak{g},\ell,\alpha}:=H_{\ell,
\alpha}/\hat{\mathfrak{g}}(X)H_{\ell, \alpha}
$
and its dual
\[
\mathbb{V}_{X,\mathfrak{g},\ell, \alpha}^{\dag}:=\mathrm{Hom}_{\mathbb{C}}(\mathbb{V}_{X,\mathfrak{g},\ell,\alpha},\mathbb{C})
=\mathrm{Hom}_{\hat{\mathfrak{g}}(X)}(H_{\ell,\alpha},\mathbb{C}),
\] which is called the space of conformal blocks.

If the $n$-pointed stable curve $X=(C,\{x\}_{x\in I})$ is singular, then the following Factorization Theorem holds.
\begin{thm}[see Theorem 3.19 of \cite{MR2433154}]\label{factorization theorem} $(1)$ Let $\overline{\mathcal{M}}_{g-1,n+2}\xrightarrow{f} \overline{\mathcal{M}}_{{g,n}}$ be the gluing map of the last two marked points $\{x_1,x_2\}$. We have $$f^{\ast}\big(\mathbb{V}^{\dagger}_{X,\mathfrak{g},\ell,\alpha}\big)\cong \bigoplus_{\mu}\mathbb{V}^{\dagger}_{X^{+},\mathfrak{g},\ell,\alpha\cup \{\mu,\mu^{\ast}\},}$$
where $X=f(X')$, $X^+=(C',I'\cup \{x_1,x_2\})$ with $X'=(C',I')$, $\mu$, $\mu^{\ast}$ are the weights at $\{x_1,x_2\}$, and the sum runs through all the $\mu$ satisfying $(\mu,\theta)\leq \ell$.

$(2)$ Let $I_1\sqcup I_2=I$ be a partition, and $\overline{\mathcal{M}}_{g_1,|I_1|+1}\times \overline{\mathcal{M}}_{g_2,|I_2|+1}\xrightarrow{h} \overline{\mathcal{M}}_{g,|I|}$ be the gluing map of the last two marked points $\{x_1,x_2\}$. We have $$h^{\ast}\big(\mathbb{V}^{\dagger}_{X,\mathfrak{g},\ell,\alpha}\big)\cong \bigoplus_{\mu}\left(\mathbb{V}^{\dagger}_{X_1^+,\mathfrak{g},\ell,\{\alpha_x\}_{x\in I_1}\cup \{\mu\}}\otimes \mathbb{V}^{\dagger}_{X_2^+,\mathfrak{g},\ell,\{\alpha_x\}_{x\in I_2}\cup \{\mu^{\ast}\}}\right),$$
where $X=h((X_1,X_2))$ and $X_i^+=(C_i,I_i\cup \{x_i\})$ with $X_i=(C_i, I_i)$
for $i=1,2$.
\end{thm}

For further applications, we also need the following Propagation of Vacua Theorem.
\begin{thm}[see Theorem 3.15 of \cite{MR2433154}]\label{theorem Propagation}
Let $f:\overline{\mathcal{M}}_{g,n+1}\rightarrow \overline{\mathcal{M}}_{g,n}$ be the forgetful map of the last marked point. Then $f^{\ast}\mathbb{V}^{\dagger}_{X,\mathfrak{g}, \ell, \alpha}=\mathbb{V}^{\dagger}_{X^+,\mathfrak{g}, \ell, \alpha\cup \{0\}}$, where $X=(C,I)=f(X^+)$ and $X^+=(C,I\cup \{x\})$.
\end{thm}

Then we show that there is  a natural product morphism
\begin{equation}\label{product of conformal blocks}
\mathbb{V}_{X,\mathfrak{g},\ell, \alpha}^{\dag}\otimes \mathbb{V}_{X,\mathfrak{g},m, \beta}^{\dag}\to \mathbb{V}_{X,\mathfrak{g},\ell+m, \alpha+\beta}. 
\end{equation}
From the representation theoretic viewpoint, this map is constructed as follows. For each $x\in I$, the tensor product $V_{\alpha_x}\otimes V_{\beta_x}$ of $\mathfrak{g}$-representations has a unique irreducible subrepresentation isomorphic to $V_{\alpha_x+\beta_x}$. Hence there is a canonical $\mathfrak{g}$-module morphism $V_{\alpha_x+\beta_x}\to V_{\alpha_x}\otimes V_{\beta_x}$. This extends to $\hat{\mathfrak{g}}$-module morphism  
\[
H_{\ell+m, \alpha_x+\beta_x}\to H_{\ell,\alpha_x}\otimes H_{m,\beta_x},
\] where $\alpha+\beta=\{\alpha_x+\beta_x\}_{x\in I}$.
Taking tensor products over $x\in I$ yields
\[
H_{\ell+m, \alpha+\beta}\to H_{\ell,\alpha}\otimes H_{m,\beta},
\]
and its dualizing gives \[
H_{\ell,\alpha}^{\ast}\otimes H_{m,\beta}^{\ast}\to H_{\ell+m, \alpha+\beta}^{\ast}.
\]
One checks that $\hat{\mathfrak{g}}$-action is compatible; passing to quotients yields the map in (\ref{product of conformal blocks}). For details see \cite[\S 2.2]{M}.

\goodbreak

Thus, the direct sum of all conformal blocks
\begin{equation}\label{equation graded}
\mathbb{V}_{X,\mathfrak{g}}^{\dag}:=\bigoplus_{\ell,\alpha}\mathbb{V}_{X,\mathfrak{g},\ell,\alpha}
\end{equation}
acquires the structure of a commutative graded (by
$\mathbb{Z}\times \prod_{x\in I}X^{\ast}(B_x)$) $\mathbb{C}$-algebra, called the {\em algebra of conformal blocks}.
The construction is relative over any family of stable curves, and this gives a flat sheaf of algebras $\mathbb{V}_{\mathfrak{g}}^{\dag}:=\oplus_{\ell,\alpha}\mathbb{V}_{\mathfrak{g},\ell,\alpha}^{\dag}$
on $\overline{\mathcal{M}}_{g,n}$ (see \cite{TUY}).

\section{Anti-canonical bundles of moduli spaces}\label{anti-canonical bundle}\label{section fano}

In this section, we show that there is canonical parabolic data $\omega_c$ such that:
\begin{itemize}
\item[(1)] $\mathcal{M}_{G,\,\omega_c}$ is Fano for a simple and simply connected group $G$ of type $C$ or $D$ (see Proposition \ref{thm type C} $\&$ \ref{thm type D});

\item[(2)] the sheaf $\omega^{-2}_{\mathcal{M}_{G,\omega_c}}$ is ample and invertible for the type $B$ case (see $\text{Remark \ref{rmk type B}}$).
\end{itemize}
\goodbreak
We consider the standard representation for each type of classical groups. First, consider
$$\Phi:G=\mathrm{Sp}_{2m}(\mathbb{C})\hookrightarrow \mathrm{SL}(V)\subset \mathrm{GL}(V),\,\, V:=\mathbb{C}^{\oplus 2m},$$ where the Dynkin index $m_{\phi}=1$ for $\phi: \mathfrak{g}\rightarrow \mathrm{sl}(V)$ and $m_{ad}=2(m+1)$ for the adjoint representation, with respect to the normalized Killing form.
Let $$T=\{\mathrm{diag}(x_1^{-1},x_2^{-1},\ldots,x_m^{-1},x_1,x_2,\ldots, x_m)\}\subset \mathrm{Sp}_{2m}(\mathbb{C})$$ be the maximal torus. Suppose $B_x$ ($x\in I$) are Borel subgroups.

Set $\omega_c:=(\{P_x\text{=}B_x\}_{x\in I},\{{a}_x\}_{x\in I})$, ${a}_x=\frac{b_x}{m_{ad}}$, $$\aligned b_x:\mathbb{G}_m \rightarrow T\subset B_x,\,\,\,\,
k \mapsto (k^{-m},\ldots,k^{-1},k^{m},\ldots,k), \endaligned$$ 
and $\widetilde{\omega}_c:=(\{\widetilde{P}_x\}_{x\in I}, \{\widetilde{{a}}_x\}_{x\in I})$, where $$ \Phi_{\ast}(P_x,{a}_x):=(\widetilde{P}_x, \widetilde{{a}}_x)=(P_{\mathrm{SL}(V)}(\widetilde {a}_x),\Phi\circ {a}_x).
$$
Clearly, the \emph{quasi-parabolic $G$-bundle} $(E,\{s_x\}_{x\in I})$ of type $\{P_x\}_{x\in I}$
gives rise to a quasi-parabolic $\mathrm{SL}(V)$-bundle $(E_{\mathrm{SL}(V)}, \{\tilde{s}_x:=\Phi_{\ast}(s_x)\}_{x\in I})$, where $$\Phi_{\ast}(s_x): \{x\}\xrightarrow{s_x}E|_{\{x\}}/P_x\hookrightarrow E_{\mathrm{SL}(V)}|_{\{x\}}/\widetilde{P}_x.$$

\goodbreak
\begin{lem}\label{lem finite map type C} Let $\tau_x=\mathrm{Lie}({a}_x)$ and $\tau'_x=\mathrm{Lie}(\widetilde{{a}}_x)$ for $x\in I$. Then $\theta_{\mathrm{sl}(V)}(\tau'_x)<1$, and there exists a finite map $\psi:\mathcal{M}_{G,\,\omega_c}\rightarrow \mathcal{M}_{\mathrm{SL}(V),\widetilde{\omega}_c}.$ 
\end{lem}
\begin{proof} By definition, $\theta_{\mathrm{sl}(V)}(\tau'_x)=\frac{m}{m+1}$, and there is a commutative diagram (see Theorem \ref{correspondence})
\begin{equation}\label{commutative diagram}
\CD
  \mathrm{Bun}_G^{A,\vec \tau}(\widehat{C}) @>>> \mathrm{Bun}_{\mathrm{SL}(V)}^{A,\vec \tau'}(\widehat{C}) \\
  @V \wr VV   @V \wr VV  \\
  \mathrm{Bun}_{G,\{P_x\}_{x\in I}} @>>> \,\mathrm{Bun}_{\mathrm{SL(V)},\{\widetilde{P}_x\}_{x\in I}},
\endCD
\end{equation} where $\vec \tau=\{\bar \tau_x\}_{x\in I}$ with $\tau_x=\frac{\bar \tau_x}{2(m+1)}$, and $\vec \tau'=\{\bar \tau'_x\}_{x\in I}$ with $\tau_x'=\frac{\bar \tau'_x}{2(m+1)}$.
Consider the faithful representation $\Phi: G\rightarrow \mathrm{SL}(V)$. Then $\widehat E_{\mathrm{SL}(V)}$ is  semistable if $\widehat E$ is semistable. Hence the diagram (\ref{commutative diagram}) restricts to the semistable locus, and we obtain the following diagram of moduli spaces \begin{equation}\label{commutative diagram of moduli space}
\CD
  \mathcal{M}_G^{A,\vec \tau}(\widehat{C}) @>>> \mathcal{M}_{\mathrm{SL}(V)}^{A,\vec \tau'}(\widehat{C}) \\
  @V\wr VV  @V\wr VV  \\
  \mathcal{M}_{G,\,\omega_c} @>\psi>> \mathcal{M}_{\mathrm{SL(V)},\widetilde \omega_c}.
\endCD
\end{equation}
To prove that $\psi$ is finite, by properness of $\mathcal{M}_{G,\, \omega_c}$, we only need to show it is affine. In fact, this follows from (\ref{equation affine}).  We are done.
\end{proof}

Recall that a
parabolic $\mathrm{SL}(V)$-bundle corresponds to a vector bundle $E$ of rank $2m$ with a (full) flag of quotients ($x\in I$)
$$E_x\twoheadrightarrow Q_{2m-1}(E)_x \twoheadrightarrow \cdots \twoheadrightarrow Q_1(E)_x\twoheadrightarrow 0,$$
and the data $\widetilde{\omega}_c$ corresponds to a sequence of integers 
$$d_1(x)=\cdots =d_{m-1}(x)=d_{m+1}(x)=\cdots=d_{2m-1}(x)=1,\,\, d_m(x)=2,$$ with 
$k=2(m+1)$ (see Example \ref{ex 2.7}). Let $\sE'$ be the universal $\mathrm{SL}(V)$-bundle on $C\times \mathrm{Bun}_{\mathrm{SL}(V),\{\widetilde{P}_x\}_{x\in I}}$. 
\begin{lem}[see Theorem 3.1 of \cite{Su3}]\label{lem 3.2} If $2m\mid |I|$, then there exists an ample line bundle $\Theta_{\mathcal{M}_{\mathrm{SL}(V),\widetilde{\omega}_c}}$
on ${\mathcal{M}_{\mathrm{SL}(V),\widetilde{\omega}_c}}$ such that \begin{equation}\label{equation pull-back of theta bundle}\pi^{\ast}_{\mathrm{SL}(V)}\Theta_{\mathcal{M}_{\mathrm{SL}(V),\widetilde{\omega}_c}}=\mathrm{det}\mathrm{R}pr_{2\ast}
\sE'(V)^{-k}\otimes \underset{x\in I}{\bigotimes}\sL(k\chi_{\widetilde{{a}}_x}), \,{\Small {k{=}2(m{+}1)}},
\end{equation}
where $\mathrm{Bun}^{ss}_{\mathrm{SL}(V),\widetilde{\omega}_c}\xrightarrow{\pi_{\mathrm{SL}(V)}}\mathcal{M}_{\mathrm{SL}(V),\widetilde{\omega}_c}, \,C\times \mathrm{Bun}^{ss}_{\mathrm{SL}(V),\widetilde{\omega}_c}\xrightarrow{pr_2}\mathrm{Bun}^{ss}_{\mathrm{SL}(V),\widetilde{\omega}_c}$.
\end{lem}
On $\mathcal{M}_{G,\, \omega_c}$, define $\Theta_{\mathcal{M}_{G,\,\omega_c}}:=\psi^{\ast}\Theta_{\mathcal{M}_{\mathrm{SL}(V),\widetilde{\omega}_c}}$.
A normal projective variety $X$ is called {Fano} if $\omega^{-1}_X=\mathcal{H}om(\omega_X,\sO_X)$ is
an ample line bundle.

\begin{prop}\label{thm type C} Assume that $2m\mid |I|\gg 0$ such that $$\mathrm{Codim}(\mathrm{Bun}_{G,\,\omega_c}\setminus \mathrm{Bun}^{rs}_{G,\,\omega_c})\geq 2\,\,\,\,(\text{see Proposition \ref{prop codim semistable and stable parabolic} $\&$ \ref{prop codimension regularly stable}}).$$ Then the moduli  space $\mathcal{M}_{G,\,\omega_c}$
is {Fano} when $G=\mathrm{Sp}_{2m}(\mathbb{C})$.
\end{prop}

\begin{proof} By the codimension assumption, it suffices to show that
\begin{equation}\label{equatin isomorphism}\omega^{-1}_{\mathcal{M}_{G,{\omega}_c}}|_{\mathcal{M}_{G,\, \omega_c}^{\mathrm{rs}}}= \Theta_{\mathcal{M}_{G,{\omega}_c}}|_{{\mathcal{M}_{G,\, \omega_c}^{\mathrm{rs}}}}.
\end{equation}

To prove (\ref{equatin isomorphism}), note that $\pi_G: \mathrm{Bun}^{rs}_{G,\, \omega_c}\rightarrow \mathcal{M}^{rs}_{G,\, \omega_c}$ is a $Z(G)$-gerbe,
it is enough to show, on $\mathrm{Bun}^{rs}_{G,\, \omega_c}$, that
\begin{equation}\label{equation line bundle isomorphism}
\omega^{-1}_{\mathrm{Bun}^{rs}_{G,\, \omega_c}}= \pi^{\ast}_G (\Theta_{\mathcal{M}_{G,\omega_c}}|_{\mathcal{M}^{rs}_{G,\omega_c}})\big(= (\psi|_{\mathcal{M}^{rs}_{G,\omega_c}} \circ \pi_G)^{\ast}\Theta_{\mathcal{M}_{\mathrm{SL}(V),\widetilde{\omega}_c}}\big).
\end{equation}
Let $\sE$ (resp. $\{S_x\}_{x\in I}$) be the universal $G$-bundle (resp. universal sections) in (\ref{4.2}). Then $\psi|_{\mathcal{M}_{G,\omega_c}^{rs}}\circ \pi_{G}:\mathrm{Bun}^{ss}_{G,\,\omega_c}\rightarrow \mathcal{M}_{G,\,\omega_c}\rightarrow \mathcal{M}_{\mathrm{SL}(V),\widetilde{\omega}_c}$
is given by $\sE_{\mathrm{SL}(V)}$ with
$\{\widetilde{S}_x:\mathrm{Bun}_{G,\omega_c}^{ss}\xrightarrow{S_x}\sE|_x/P_x{\hookrightarrow} \sE_{\mathrm{SL}(V)}|_x/\widetilde{P}_x\}_{x\in I}$, and there is a commutative diagram
$$\CD
  \mathrm{Bun}_{G,\, \omega_c}^{rs} @>\bar \Phi>> \mathrm{Bun}^{ss}_{\mathrm{SL}(V),\widetilde{\omega}_c} \\
  @V\pi_G VV @V\pi_{\mathrm{SL}(V)} VV  \\
  \mathcal{M}^{rs}_{G,\,\omega_c} @>\psi|_{\mathcal{M}^{rs}_{G,\, \omega_c}} >> \mathcal{M}_{\mathrm{SL}(V),\widetilde{\omega}_c}.
\endCD
$$
Combining with Lemma \ref{lem 3.2}, $\pi^{\ast}_G(\Theta_{\mathcal{M}_{G,\omega_c}}^{m_{ad}}|_{\mathcal{M}^{rs}_{G,\omega_c}})=\bar{\Phi}^{\ast}\pi_{\mathrm{SL}(V)}^{\ast}\Theta^{m_{ad}}_{\mathcal{M}_{\mathrm{SL}(V),\widetilde{\omega}_c}}\\
$
$$\aligned =&\bar{\Phi}^{\ast}(\mathrm{detR}pr_{2\ast}\sE'(V))^{-km_{ad}}\otimes \underset{x\in I}{\bigotimes}\, \bar{\Phi}^{\ast}\sL(k\chi_{\widetilde{{a}}_x})^{m_{ad}} \\=&\mathcal{D}_{\mathrm{Ad}}^{km_{\phi}}\otimes \underset{x\in I}{\bigotimes}\, \bar{\Phi}^{\ast}\sL(k\chi_{\widetilde{{a}}_x})^{m_{ad}}
\endaligned$$ by a result of Kumar, Narasimhan and  Ramanathan $$\mathcal{D}_{\mathrm{Ad}}^{m_{\phi}}= \mathrm{detR}pr_{2\ast}\sE(V)^{-m_{ad}}\,\,\,\,\text{(see \cite[Theorem 5.4]{KNR})},$$
where $m_{\phi}=1$, $k=m_{ad}=2(m+1)$. Recall the formula of the anti-canonical bundle in Lemma \ref{lem anti-canonical}. Then
(\ref{equation line bundle isomorphism}) reduces to showing that
$$
\bar{\Phi}^{\ast}\sL(\chi_{\widetilde{{a}}_x})^{m_{ad}}=\sL(2\rho)^{m_{\phi}},\,\, 2\rho=m_{ad}\chi_{{a}_x},
$$
which follows from the formula ${\bar{\Phi}}_x^{\ast}\sE'(\chi_{\widetilde{{a}}_x})^{-1}=\sE(m_{\phi}\chi_{{a}_x})^{-1}$ with
$\bar{\Phi}_x:\sE_x/P_x\rightarrow \sE'_x/\widetilde{P}_x$.
To see this, note that $\bar{\Phi}_x^{\ast}\sE'(\chi_{\widetilde{{a}}_x})^{-1}$
is defined by $$\Phi_x^{\ast}(\chi_{\widetilde{{a}}_x})^{-1}: P_x\xrightarrow{\Phi}\widetilde{P}_x\xrightarrow{(\chi_{\widetilde{{a}}_x})^{-1}}
\mathbb{G}_m,\,\,\, t\mapsto (\chi_{\widetilde{{a}}_x}) (\Phi(t))^{-1},$$
then by (\ref{2.2}), for any 1-PS $\lambda'$, we have
$$\Phi_x^{\ast}(\chi_{\widetilde{{a}}_x})(\lambda')=({a}_x,\lambda')_{\mathfrak{g}}
=m_{\phi}(\widetilde{{a}}_x, \Phi_{\ast}(\lambda'))_{\mathrm{sl}(V)}=m_{\phi}\chi_{{a}_x}(\lambda'),$$
where ${a}_x$ is a 1-PS (not only rational). We are done.
\end{proof}

Second, consider $G=\mathrm{Spin}_{2m}(\mathbb{C})$, and $$\Phi:G\xrightarrow{\pi} \mathrm{SO}_{2m}(\mathbb{C}) \overset{\Phi'}{\hookrightarrow} \mathrm{SL}(V)\subset \mathrm{GL}(V),\,\, V:=\mathbb{C}^{\oplus 2m}$$ factors through a finite cover $\pi$, where the Dynkin index $m_{\phi}=2$
for ${\phi: \mathrm{so}_{2m}(\mathbb{C})\rightarrow \mathrm{sl}(V)}$ and $m_{ad}=2(2m-2)$, with respect to the normalized Killing form. Fix a maximal torus $$T=\{\mathrm{diag}(x_1^{-1},x_2^{-1},\cdots, x^{-1}_m, x_1,x_2,\cdots, x_m)\}\subset \mathrm{SO}_{2m}(\mathbb{C})\hookrightarrow \mathrm{SL}(V),$$ and set $B_x\subset \mathrm{SO}_{2m}(\mathbb{C})$ $(x\in I)$ to be $\text{Borel subgroups}$.

Since the finite cover $\pi$ induces an isomorphism of Lie algebras, we denote by $\omega=(\{P_x\}_{x\in I}, {a})$ for both \emph{parabolic $G$-bundles} and \emph{parabolic $\mathrm{SO}_{2m}(\mathbb{C})$-bundles}. Set $\omega_c:=(\{P_x=B_x\}_{x\in I}, \{{a}_x\}_{x\in I})$, ${a}_x=\frac{b_x}{m_{ad}}$,
$$\aligned b_x:\mathbb{G}_m\rightarrow T\subset B_x,\,
k \mapsto (k^{-2m+2},k^{-2m+4},\ldots,1,k^{2m-2},k^{2m-4},\ldots,1), \endaligned$$
and
$\widetilde{\omega}_c:=(\{\widetilde{P}_x\}_{x\in I}, \{\widetilde{{a}}_x\}_{x\in I})$ where $$\Phi'_{\ast}(P_x, {a}_x):=(\widetilde{P}_x,\widetilde{{a}}_x)=(P_{\mathrm{SL}(V)}(\widetilde {a}_x),\Phi'\circ {a}_x).$$ Let $\tau_x=\mathrm{Lie}({a}_x)$ and $\tau'_x=\mathrm{Lie}(\widetilde{{a}}_x)$. By definition, $\theta_{\mathrm{sl}(V)}(\tau'_x)=1$.
Consider the morphism $$\bar{\Phi}:\mathrm{Bun}_G^{A,\vec \tau}(\widehat{C})\rightarrow \mathrm{Bun}_{\mathrm{SO}_{2m}(\mathbb{C})}^{A,\vec \tau}(\widehat{C})\rightarrow \mathrm{Bun}_{\mathrm{SL}(V)}^{A,\vec \tau'}(\widehat{C}),$$
where $\vec \tau_x=\{\bar \tau_x\}_{x\in I}$ with $\tau_x=\frac{\bar \tau_x}{2m-2}$, and $\vec \tau'_x=\{\bar\tau'_x\}_{x\in I}$ with $\tau'_x=\frac{\bar \tau'_x}{2m-2}$. As described in the type $C$ case, there is a finite map
\begin{equation}\label{equation finite map} \psi: \mathcal{M}_{G,\,\omega_c}\cong \mathcal{M}_G^{A,\vec \tau}(\widehat C)
\xrightarrow{\epsilon} \mathcal{M}^{A,\vec \tau}_{\mathrm{SO}_{2m}(\mathbb{C})}(\widehat C)\rightarrow \mathcal{M}^{A,\vec \tau'}_{\mathrm{SL}(V)}(\widehat C),\end{equation}
where $\epsilon$ is affine by a result of Ramanathan (see \cite{Ramana}). However, $\text{Theorem
\ref{correspondence}}$ does not apply here since $\tau'_x\in \Phi_0-\Phi_0^s$ (for $\mathrm{SL}(V)$). To remedy this, we refer to the parahoric subgroup of $\mathrm{GL}(K_x)$ with $K_x=\mathbb{D}^{\ast}_x$ for $x\in I$ ($\text{see \cite{B.C} for the general case}$) as follows.

Let $A_x:=\mathbb{C}[[z]]$, $K_x:=\mathbb{C}((z))$ for $x\in I$, and ${m_r(\widetilde {a}_x):=-\lceil <\widetilde a_x, r> \rceil}$, where $r\in R$ (the root system) and $\lceil\cdot \rceil$ is the ceiling function. Recall that $\pi:\hat C\rightarrow C$ is a Galois cover with a finite Galois group $A$. 
\begin{defn}
The parahoric subgroup $\mathrm{GL}_{\tau'_x}\subset \mathrm{GL}(K_x)$ is $\text{defined as}$ $$\mathrm{GL}_{\tau'_x}:=\langle T(A_x), U_r(z^{m_r(\widetilde {a}_x)}A_x),\,r\in R\rangle\,\,\,\text{for}\,x\in I.$$
\end{defn}
By \cite[Proposition 3.1.1]{B.C}, there is an isomorphism \begin{equation}\label{uniformalization}
{\Tiny \mathrm{Bun}_{\mathrm{GL}(V)}^{A,\vec \tau'}(\widehat{C})\cong \left[\prod_{x\in I} \mathrm{GL}_{\tau'_x}\setminus \prod_{x\in I} \mathrm{GL}(K_x)/ \mathrm{GL}(C-\{x\}_{x\in I}) \right].}
\end{equation}
Since $m_r(\widetilde {a}_x)=1$ for some $r\in R$, the subgroup $\mathrm{GL}_{\tau'_x}\subset \mathrm{GL}(K_x)$ is not contained in $\mathrm{GL}(A_x)$. We will show that it is conjugate to some subgroup of $\mathrm{GL}(A_x)$.

Let $z$ (resp. $\omega$) be the local coordinate of $x$ (resp. $\hat{x}\in \pi^{-1}(x)$) with $z=\omega^{d_x}$ for $d_x:=2m-2$. Let $\gamma$ be a generator of $A_{\hat{x}}$, and $\zeta$ be a primitive $d_x$-th root of unity. Then the isotropy group $A_{\hat{x}}$ acts on $\omega$ via $\gamma\cdot \omega=\zeta\omega$.
Under this setting, $\tau'_x$ corresponds to the homomorphism $$\theta_{\hat{x}}: A_{\hat{x}}\rightarrow T\subset \mathrm{GL}(V),\,\,
\gamma \mapsto \mathrm{diag}(\zeta^{\text{-}(m\text{-}1)},\ldots,\zeta^{\text{-}1},1,1,\zeta^{1},\ldots,\zeta^{m\text{-}1}),
$$ and the group $\mathrm{GL}_{\tau'_x}$ consists of $(a_{ij})_{2m \times 2m}$ satisfying
$${\tiny
a_{2m,1}\in z^{-1}A_x, a_{m,m+1},
a_{ij}\in A_x\,\,\text{for}\, i\geq j \,\text{and}\, a_{ij}\in zA_x\,\,\text{for others}}.$$
\begin{lem}\label{lem 3.5} There is an isomorphism \begin{equation}\label{uniformization}{\Tiny \aligned\left[\prod_{x\in I}\mathrm{GL}_{\tau'_x}\setminus \prod_{x\in I} \mathrm{GL}(K_x)/ \mathrm{GL}(C\text{-}\{x\}_{x\in I}) \right]&\rightarrow \left[\prod_{x\in I}\mathrm{GL}'_{\tau'_x}\setminus \prod_{x\in I} \mathrm{GL}(K_x)/ \mathrm{GL}(C\text{-}\{x\}_{x\in I}) \right] \\
\bar{\Theta}_x&\mapsto \bar{\Theta}_x':=\overline{\mathrm{diag}( \overbrace{z^{-1},\ldots, z^{-1}}^{m-1},\overbrace{1,\ldots, 1}^{m+1})\cdot \Theta_x}\endaligned}\end{equation} where, for $x
\in I$, we have  $ \mathrm{GL}(K_x)\supset\mathrm{GL}'_{\tau'_x}=$
 $$\aligned \mathrm{diag}( \overbrace{z^{-1},\ldots, z^{-1}}^{m-1},\overbrace{1,\ldots, 1}^{m+1})\cdot \mathrm{GL}_{\tau'_x}\cdot \mathrm{diag}( \overbrace{z,\ldots, z}^{m-1},\overbrace{1,\ldots, 1}^{m+1}).\endaligned$$
\end{lem}
\begin{proof} The proof is a direct computation using the Uniformization  Theorem (see (\ref{uniformalization})). We omit the details.
\end{proof}

Furthermore, after a change of basis,
$\mathrm{GL}'_{\tau'_x}$ becomes $${\Tiny
\begin{pmatrix}
0& I_{m+1}\\
I_{m-1}&0
\end{pmatrix}\mathrm{GL}'_{\tau'_x}\begin{pmatrix}
0& I_{m+1}\\
I_{m-1}&0
\end{pmatrix}=ev_0^{-1}(P'_x)},
$$ where $ev_0: \mathrm{GL}[[z]]\rightarrow \mathrm{GL}(V)$, and the standard parabolic subgroup $P'_x\subset \mathrm{GL}(V)$ is of type $\vec{\,r}'(x)=(2,3,\ldots, m,\widehat{m+1},m+2,\ldots, 2m-1)$.
Let $\vec a'(x)=(0,1,\ldots, 2m-3)$ for $x\in I$. We have
\begin{prop}[see \cite{M.C}]\label{semistable correspondence} There exists an isomorphism $$\mathrm{Bun}_{\mathrm{GL}(V)}^{A,\vec \tau'}(\widehat{C})\cong \mathrm{Bun}_{\mathrm{GL}(V),\{P'_x\}_{x\in I}}.$$
Under the isomorphism, $A$-semistable $($resp. $A$-stable$)$ bundles of degree $e$ on $\widehat{C}$ correspond to {$\omega_c^{(m\text{-}1)}$-semistable} $($resp. {$\omega_c^{(m\text{-}1)}$-stable}$)$ parabolic bundles of degree $e\text{-}(m\text{-}1)|I|$ on $C$, where $\omega_c^{(m\text{-}1)}=(k, \{\vec{\,r}'(x),\vec a'(x)\}_{x\in I})$ and $k=2m-2$.
\end{prop}

\begin{prop}\label{thm type D} Assume that $2m\mid |I|\gg 0$ such that $$\mathrm{Codim}(\mathrm{Bun}_{G,\,\omega_c}\setminus \mathrm{Bun}^{rs}_{G,\,\omega_c})\geq 2\,\,\,\,(\text{see Proposition \ref{prop codim semistable and stable parabolic} $\&$ \ref{prop codimension regularly stable}}).$$ Then the moduli space $\mathcal{M}_{G,\, \omega_c}$
is {Fano} when $G=\mathrm{Spin}_{2m}(\mathbb{C})$.
\end{prop}
\begin{proof} Recall that, by (\ref{equation finite map}) and Proposition \ref{semistable correspondence}, there is a finite map $$\psi:\mathcal{M}_{G,\, \omega_c}\rightarrow \mathcal{M}^{A,\vec \tau'}_{\mathrm{GL}(V)}(\hat C)\cong \mathcal{M}_{\mathrm{GL}(V),\, \omega_c^{(m\text{-}1)}}.$$ Analogously to Lemma \ref{lem 3.2}, there is an ample line bundle $\Theta_{\mathcal{M}_{\mathrm{GL}(V)},\omega_c^{(m\text{-}1)}}$ such that $\pi^{\ast}_{\mathrm{GL}(V)}\Theta_{\mathcal{M}_{\mathrm{GL}(V)},\omega_c^{(m\text{-}1)}}=$ $$\mathrm{det}\mathrm{R}pr_{2
\ast} \sE''(V)^{\text{-}(2m\text{-}2)}\otimes \bigotimes_{x\in I}\bigotimes_{i=1}^{l_x}\mathrm{det}(\mathcal{Q})^{d'_i(x)}_{x,i}:=\Theta({\sE''}),$$
where $\pi_{\mathrm{GL}(V)}:\mathrm{Bun}^{rs}_{\mathrm{GL}(V),\, \omega_c^{(m\text{-}1)}}\rightarrow \mathcal{M}^{rs}_{\mathrm{GL}(V),\,\omega_c^{(m\text{-}1)}}$, $\sE''(V)$ is the universal bundle on $C\times \mathrm{Bun}_{\mathrm{GL}(V),\,\omega_c^{(m\text{-}1)}}$, and $$\sE''(V)_x\twoheadrightarrow \sQ_{x,l_x}\twoheadrightarrow \cdots \twoheadrightarrow \sQ_{x,1}\twoheadrightarrow 0 \,\,\text{of type $\vec{\,r}'(x)$}$$
are universal quotients on $\mathrm{Bun}_{\mathrm{GL}(V),\,\omega_c^{(m\text{-}1)}}$, ${\Small \omega_c^{(m\text{-}1)}=(k,\{\vec{\,r}'(x),\vec a'(x)\}_{x\in I})}$.

Now, we go to show that $\omega^{-1}_{\mathcal{M}_{G,\omega_c}}=\psi^{\ast}{\Theta_{\mathcal{M}_{\mathrm{GL}(V)},\omega_c^{(m\text{-}1)}}}$
by making a modification of the argument in the proof of Proposition \ref{thm type C}.
Let $\mathrm{Bun}_C(2m,0,\widetilde{\omega}_c)$ be the stack of parabolic bundles (of rank $2m$ and $\text{degree $0$}$) on $C$, and $$\{H_x^{(m\text{-}1)}\}_{x\in I}:\mathrm{Bun}_C(2m,0,\widetilde{\omega}_c)\rightarrow \mathrm{Bun}_C(2m,-(m-1)|I|,\,\omega_c^{(m\text{-}1)})$$ be the $(m-1)$-times Hecke transformation at $\{x\}_{x\in I}$ (see \cite{sz2}). We {\bf claim} that there is a commutative diagram
$${\tiny
\CD
  \mathrm{Bun}_{G,\,\omega_c} @>\bar \Phi >> \overset{\mathrm{Bun}_C(2m,0,\widetilde \omega_c)}{(\subset \mathrm{Bun}_{\mathrm{GL}(V),\widetilde{\omega}_c})} @>\{H_x^{(m\text{-}1)}\}_{x\in I}>> \mathrm{Bun}_C(2m,-|I|(m-1),\,\omega_c^{(m-1)}) @. \\
  @V\wr VV @.\ @V \cap VV  \\
  \mathrm{Bun}_{G}^{A,\vec \tau}(\widehat{C}) @>\hat{\Phi}>> \mathrm{Bun}_{\mathrm{GL}(V)}^{A,\vec \tau'}(\widehat{C})@>\,\,\,\,\,\,\,\,\ \ \ \sim \ \ \ \,\,\,\,\,\,\,\,\,>> \,\mathrm{Bun}_{\mathrm{GL}(V),\{P_x'\}_{x\in I}},
\endCD}
$$ where $\{H_x^{(m\text{-}1)}\}_{x\in I}^{\ast}\Theta(\sE'')=\Theta(\sE')$ by a direct computation, and $\sE'$ is the universal bundle on $C\times \mathrm{Bun}_C(2m,0,\widetilde \omega_c) $.
If the claim holds, then on $\mathrm{Bun}^{rs}_{G,\,\omega_c}$, $$\aligned \pi_{G}^{\ast}\psi^{\ast}\Theta_{\mathcal{M}_{\mathrm{GL}(V),\omega_c^{(m\text{-}1)}}}^{{m_{ad}}}=&\bar \Phi^{\ast} \Theta(\sE')^{m_{ad}}\\
=\bar{\Phi}^{\ast}(\mathrm{detR}pr_{2\ast}\sE'(V))^{-km_{ad}}\otimes \underset{x\in I}{\bigotimes}\, &\bar{\Phi}^{\ast}\sL(k\chi_{\widetilde{{a}}_x})^{m_{ad}},\,\,k=2m-2 \endaligned$$ (see Example \ref{ex 2.7}). This implies the result.

Finally, the claim is true, since $\bar \Phi\circ\{H^{(m\text{-}1)}_x\}_{x\in I}$ is given by $${\tiny \aligned\left[\prod_{x\in I}\mathcal{P}_{\tau_x}(K_x)\setminus \prod_{x\in I} G(K_x)/ G(C-\{x\}_{x\in I}) \right]&\rightarrow \left[\prod_{x\in I}\mathrm{GL}'_{\tau'_x}\setminus \prod_{x\in I} \mathrm{GL}(K_x)/ \mathrm{GL}(C-\{x\}_{x\in I}) \right] \\
\bar{g}_x&\mapsto \overline{\mathrm{diag}( \overbrace{z^{-1},\ldots, z^{-1}}^{m-1},\overbrace{1,\ldots, 1}^{m+1})\cdot \hat{g}_x}\endaligned}$$
(see Lemma \ref{lem 3.5}), where $\hat g_x\in \mathrm{GL}(K_x)$ is the image of $ g_x$ under the morphism ${G(K_x)\rightarrow \mathrm{GL}(K_x)}$ and $\sP_{\tau_x}(K_x):=ev_0^{-1}(P_x)\subset G(A_x)$.
\end{proof}
\begin{rmk}\label{rmk type B} Similarly, for $G=\mathrm{Spin}_{2m+1}(\mathbb{C})$ and $2m+1\mid |I|\gg 0$, we can only show that   $\omega^{-2}_{\mathcal{M}_{G,\omega_c}}$ is an ample line bundle. The reason is that $m_{\phi}=2$ for the standard representation and $m_{ad}=2(2m-1)$, but $(2m-1)\chi_{\widetilde{{a}}_x}$ is not a character $-$ only a rational one.
\end{rmk}

\section{Fano type of moduli spaces}
A normal projective variety $X$ over $\mathbb{C}$ is of Fano type, if there exists an effective $\mathbb{Q}$-divisor $\Delta$ such that $(X,\Delta)$ is klt and $-(K_X+\Delta)$ is nef and big (see \cite[Lemma-Definition 2.6]{PS}). Our goal in this section is to show
\begin{thm}\label{thm Fano type of moduli space} Let $C$ be a complex smooth projective curve of genus $g\geq 3$, and $\mathcal{M}_{G,\omega}$ be the moduli space of $\omega$-semistable {\em parabolic $G$-bundles} on $C$.
Then $\mathcal{M}_{G,\omega}$ is of Fano type when $G$ is a classical simple and simply connected group.
\end{thm}

To simplify notation, we assume from now on that $G$ is a classical simple and simply connected group.

\begin{defn}\label{defn cox} Let $X$  be a normal projective variety with finitely generated $\mathrm{Cl}(X)$, and $\Gamma$ be a finitely generated group of Weil divisors on $X$ such that $\Gamma_{\mathbb{Q}}\xrightarrow{\sim} \mathrm{Cl}(X)_{\mathbb{Q}}$. We say that $X$ is called a Mori dream space if the Cox ring
$$\mathrm{Cox}(X):=\bigoplus_{D\in \Gamma}\mathrm{H}^0(X,\sO_X(D))$$
is a $\mathbb{C}$-algebra of finite type. When $X$ is $\mathbb{Q}$-factorial, this coincides with the usual notion of a Mori dream space (see \cite[Proposition 2.9]{HK}).
\end{defn}

\begin{lem}\label{lem well known}
Let $X$ and $Y$ be two normal projective varieties over $\mathbb{C}$ that are isomorphic in codimension one. If $X$ is of Fano type,
then so is $Y$.
\end{lem}
\goodbreak
\begin{proof}
Since $X$ is of Fano type, it follows from Corollary 5.3 $\&$ 5.4 of \cite{GOST} that $X$ is a Mori dream space of globally F-regular type. Moreover, there is a common open subset $U\subset X,Y$ of codimension at least two. Thus, $Y$ is also a Mori dream space of globally F-regular type, due to that
$$\mathrm{Cl}(X)=\mathrm{Cl}(U)=\mathrm{Cl}(Y),\,\,\, \mathrm{Cox}(X)=\mathrm{Cox}(U)=\mathrm{Cox}(Y).$$ Then the result follows from  \cite[Corollary 5.4]{GOST}.
\end{proof}

We remark that when $\mathrm{Codim}(\mathrm{Bun}_{G,\omega}\setminus \mathrm{Bun}_{G,\omega}^{rs})\geq 2$,
\begin{equation}\label{equation free}\mathrm{Cl}(\mathcal{M}_{G,\omega})=\mathrm{Cl}(\mathcal{M}^{rs}_{G,\omega})=\mathrm{Pic}(\mathcal{M}^{rs}_{G,\omega})
\hookrightarrow \mathrm{Pic}(\mathrm{Bun}_{G,\omega})\end{equation} is a free abelian group of finite rank.
Set $\mathrm{Cl}(\mathcal{M}_{G,\omega}):=\underset{1\leq i\leq s}{\oplus}\mathbb{Z}\cdot D_i$. We have
\begin{equation}\label{equation cox}
\mathrm{Cox}(\mathcal{M}_{G,\omega})=\bigoplus_{\Tiny (n_1,\ldots,n_s)\in \mathbb{Z}^{\oplus s}}\mathrm{H}^0(\mathcal{M}_{G,\omega},\sO(n_1D_1+\cdots +n_s D_s)).
\end{equation}

\begin{proof}[Proof of Theorem \ref{thm Fano type of moduli space}]
The argument is divided into two steps.

{\bf Step I}: We consider the case $\omega=(\{B_x\}_{x\in I},{a})$, where $B_x\subset G$ are Borel subgroups such that $$(\ast)\,\,\,\,\,\,\,\,\,
\aligned &2m\mid |I| \,\,\,\,\,\,\,\, \,\,\,\,\,\,\text{for}\,\, G=\mathrm{Sp_{2m}}(\mathbb{C}),\mathrm{Spin}_{2m}(\mathbb{C}), \\ \,\, &2m+1\mid |I|\,\,\,\, \text{for}\,\, G=\mathrm{Spin}_{2m+1}(\mathbb{C}).\endaligned$$ 

Note that the $\mathbb{Q}$-Cartier divisor $-K_M$ of 
$$M:=\mathcal{M}_{G,\,\omega_c},\quad \omega_c=(\{B_x\}_{x\in I},\{{a}_x\}_{x\in I})$$ 
is ample by Proposition \ref{thm type C} $\&$ \ref{thm type D} and $\text{Remark \ref{rmk type B}}$. We {\bf claim} that there exists a moduli space $M':=\mathcal{M}_{G,\bar \omega}$ with klt singularities, and a birational morphism $f:M'\rightarrow M$ that is isomorphic in codimension one. If the claim holds, then $M'$ is of Fano type since $-K_{M'}=f^{\ast}(-K_M)$ is nef and big. Furthermore,  $\mathcal{M}_{G,\omega}$ is of Fano type by Lemma \ref{lem well known}, due to
$$ \mathrm{Codim}(\mathrm{Bun}_{G,\omega}\setminus \mathrm{Bun}_{G,\omega}^{s})\geq 2         \,\,\,\,\text{(when\,$g\geq 3$,\, see Remark \ref{rmk codim stable principal})}$$ for arbitrary parabolic data $\omega$.

To prove the claim, let $\bar \omega=(\{B_x\}_{x\in I}, \{\bar {a}_x\}_{x\in I})$, where $\{\bar {a}_x\}_{x\in I}$
is generic after a minor modification of $\{{a}_x\}_{x\in I}$ and $\theta_{\mathrm{sl}(V)}(\bar \tau_x')<1$ (see Section \ref{section fano}) for $x\in I$. Then there is a birational morphism $$f:\mathcal{M}_{G,\bar \omega}:=M'\rightarrow M:=\mathcal{M}_{G,\,\omega_c},$$ 
and it suffices to show that $M'$ has klt singularities. Recall that
$$\mathcal{M}_{G,\bar \omega}=\mathfrak{F}^s_{\Phi\text{-}\mathrm{FlPsBun}}//\mathrm{GL(Y)},\,\,\,\, \mathrm{Bun}^s_{G,\bar \omega}=[\mathfrak{F}^s_{\Phi\text{-}\mathrm{FlPsBun}}/\mathrm{GL(Y)}]$$ (see \cite[\S 5.7]{Heinloth}). It follows that any open affine $U\subset \mathcal{M}_{G,\bar \omega}$ is of klt type by $\text{\cite[Theorem 5]{Lukas}}$. Thus $U$ (hence $M'$) has klt singularities since $-K_{M'}$ is $\mathbb{Q}$-Cartier, and the singular locus has codimension at least two, by $$\mathrm{Codim}(\mathrm{Bun}_{G,\bar \omega}\setminus \mathrm{Bun}^{rs}_{G, \bar \omega})\geq 2\,\,\,\text{(see Remark \ref{codim regularly stable principal})}.$$

{\bf Step II}: We consider the general case $\omega=(\{P_x\}_{x\in I},{a})$, where each $P_x\subset G$
is a parabolic subgroup. Fix one Borel subgroup $B_x\subset G$ for $x\in I\cup J$ ($B_x\subset P_x$ for $x\in I$). Then there is a natural morphism 
$$\mathrm{Bun}_{G,\{B_x\}_{x\in I\cup J}}\rightarrow \mathrm{Bun}_{G,\{P_x\}_{x\in I}}\,\,\text{(see Lemma \ref{thm forgetting map of stack})}.$$
Set $\omega_J:=(\{B_x\}_{x\in I\cup J}, b)$, where $|I\cup J|$ satisfies $(\ast)$, and $b$ is generic after a minor modification of ${a}\cup \{0\}_{x\in J}$. This induces a surjective morphism $\mathcal{M}_{G,\,\omega_J}\rightarrow \mathcal{M}_{G,\omega}$. Since $\mathcal{M}_{G, \omega_J}$ is of Fano type, we are done by \cite[Corollary 5.2]{F.T}.
\end{proof}

\begin{cor}\label{cor cox ring} When $g\geq 3$, the moduli space $\mathcal{M}_{G,\omega}$ is a Mori dream space. In particular, the Cox ring $\mathrm{Cox}(\mathcal{M}_{G,\omega})$ is finitely generated.
\end{cor}

\begin{proof}
This follows since $\mathcal{M}_{G,\omega}$ is of Fano type (see \cite[Corollary 1.3.2]{BCHM} and \cite[Corollary 5.3]{GOST}).
\end{proof}

\begin{rmk}\label{rmk g=0,1,2} If ${\mathrm{Codim}(\mathrm{Bun}_{G,\omega}\setminus \mathrm{Bun}_{G,\omega}^{rs})\geq 2}$, the proof of $\text{Theorem \ref{thm Fano type of moduli space}}$ also
implies that the moduli space $\mathcal{M}_{G,\omega}$ is of Fano type, and its Cox ring $\mathrm{Cox}(\mathcal{M}_{G,\omega})$ is finitely generated.
\end{rmk}


\section{Finite generation of the algebra of conformal blocks}
Our main result of this section is 
\begin{thm}\label{f.g. for singular}
For any $n$-pointed stable curve $X =(C,\{x\}_{x\in I})\in \overline{\mathcal{M}}_{g,n}$ and any classical simple Lie algebra $\mathfrak{g}$, the \emph{algebra of conformal blocks} $\mathbb{V}_{X,\mathfrak{g}}^\dagger$ is finitely generated.
\end{thm}
This solves Conjecture 1.4 of \cite{MY} for classical simple Lie algebras.
The proof mainly depends on the property of the Cox ring $\mathrm{Cox}(\mathcal{M}_{G,\omega})$ (see Corollary \ref{cor cox ring} and Remark \ref{rmk g=0,1,2}).
To begin, we introduce a nice description of the space of conformal blocks in terms of \emph{parabolic $G$-bundles} (see also \cite[Theorem 1.7]{B.F} for the singular case).
\begin{thm}[see Theorem 1.2.1 of \cite{L.S}]\label{thm isomorphism}
For ${X=(C,\{x\}_{x\in I}) \in {\mathcal{M}}_{g,n}}$ and $(\alpha_x,\theta)\leq \ell$ $(x\in I)$, there is a canonical isomorphism $(see\,\, \mathrm{(\ref{4.1})})$
\begin{equation}\label{equation 7.1}
\mathrm{H}^0(\mathrm{Bun}_{G,\{P_x\}_{x\in I}},\mathcal{L}(\ell,\alpha)) \xrightarrow{\sim} \mathbb{V}_{X,\mathfrak{g},\ell,\alpha}^\dagger,
\end{equation}  where the parabolic subgroup $P_x$ is determined by the weight $\alpha_x$.
\end{thm}

Let $B_x\subset P_x$ ($x\in I$) be Borel subgroups. The embedding  $$\mathrm{H}^0(\mathrm{Bun}_{G,\{P_x\}_{x\in I}},\mathcal{L}(\ell,\alpha))\hookrightarrow \mathrm{H}^0(\mathrm{Bun}_{G,\{B_x\}_{x\in I}},\mathcal{L}(\ell,\alpha'))$$
can be naturally viewed as a subspace, where the character $\chi_{\alpha_x'}:=\chi_{\alpha_x}|_{B_x}$. With this identification, there is a natural multiplication on the \emph{algebra of conformal blocks} $\mathbb{V}^{\dagger}_{X,\mathfrak{g}}$ induced by the isomorphism (\ref{equation 7.1}), which coincides with the multiplication (\ref{product of conformal blocks}) 
$$
\mathbb{V}_{X,\mathfrak{g},\ell,\alpha}^\dagger \otimes \mathbb{V}_{X,\mathfrak{g},m,\beta}^\dagger \to \mathbb{V}_{X,\mathfrak{g},\ell+m,\alpha+\beta}^\dagger
$$ when $C$ is smooth (see \cite[\S 2.3]{M}), where $\alpha+\beta:=\{\alpha_x+\beta_x\}_{x\in I}$.

The following two lemmas will be useful in the proof of Theorem \ref{f.g. for singular}.
\begin{lem}\label{useful lemma 1} Let $X=(C,\{x\}_{x\in I})$ be an $n$-pointed smooth curve, and $\omega=(\{B_x\}_{x\in I},a)$ where  $B_x\subset G$ are Borel subgroups. If
\begin{equation}\label{equation codim}
\mathrm{Codim}(\mathrm{Bun}_{G,\omega}\setminus \mathrm{Bun}^{rs}_{G,\omega})\geq 2, \end{equation} then $\mathrm{Cox}(\mathcal{M}_{G,\omega})=\mathbb{V}^{\dagger}_{X,\mathfrak{g}}$.
\end{lem}
\begin{proof} By (\ref{equation free}),  $\mathrm{Cl}(\mathcal{M}_{G,\omega})=\mathrm{Pic}(\mathcal{M}^{rs}_{G,\omega})$ is a free abelian group of finite rank, and $$\mathrm{Cox}(\mathcal{M}_{G,\omega})=\mathrm{Cox}(\mathcal{M}^{rs}_{G,\omega})=\bigoplus_{L \in \mathrm{Pic}(\mathcal{M}^{rs}_{G,\omega})} \mathrm{H}^0(\mathcal{M}^{rs}_{G,\omega},L).$$
Consider the $\mathrm{Z}(G)$-gerbe $\mathrm{Bun}^{rs}_{G,\omega}\xrightarrow{\pi_G} \mathcal{M}^{rs}_{G,\omega}$. Then any line bundle $\mathcal{L}$ with nonzero global sections on $\mathrm{Bun}^{rs}_{G,\omega}$ descends to $\mathcal{M}^{rs}_{G,\omega}$. Thus, 
$$\aligned
\mathrm{Cox}(\mathcal{M}_{G,\omega})
=&\bigoplus_{\mathcal{L} \in \mathrm{Pic}(\mathrm{Bun}^{rs}_{G,\omega})} \mathrm{H}^0(\mathrm{Bun}^{rs}_{G,\omega},\mathcal{L}) =\bigoplus_{\mathcal{L} \in \mathrm{Pic}(\mathrm{Bun}_{G,\omega})} \mathrm{H}^0(\mathrm{Bun}_{G,\omega},\mathcal{L})\\
=&\bigoplus_{\ell,\alpha} \mathrm{H}^0(\mathrm{Bun}_{G,\{B_x\}_{x\in I}},\mathcal{L}(\ell,\alpha)) =\mathbb{V}_{X,\mathfrak{g}}^\dagger \,,
\endaligned$$ where the last equality follows from Theorem \ref{thm isomorphism}. We are done. 
\end{proof}

\begin{lem}\label{useful lemma 2}
Let $A$ be a free abelian group of finite rank, and $R$ be a finitely generated $A$-graded $\mathbb{C}$-algebra with grading map $\pi:R\rightarrow A$. If $B\subset A$ is a subgroup and $R_B:=\{x\in R\mid \pi(x)\in B\}\subset R$ is the corresponding subalgebra, then $R_B$ is also finitely generated. Moreover, if $R$ is integrally closed, then so is $R_B$.
\end{lem}
\begin{proof} The finite generation part is exactly \cite[Lemma 4.11]{MY}. Moreover, its proof implies that $R_B$ can be realized as $R^H$ for some reductive group $H$ that acts naturally on $R$. Then $R_B$ is integrally closed by $\text{\cite[Proposition 3.1]{D}}$. We are done.
\end{proof}

We now go to prove Theorem \ref{f.g. for singular}. Following the case $\mathfrak{g}=\mathfrak{sl}_r$ (see \cite{MY}), we will exhibit 
the algebra $\mathbb{V}^{\dagger}_{X,\mathfrak{g}}$ as a subalgebra of a finitely generated algebra $R$ of the form $R_B$.

\begin{proof}[Proof of Theorem \ref{f.g. for singular}: the smooth case] When $g\geq 3$, we have $$\mathrm{Codim}(\mathrm{Bun}_{G,\omega}\setminus \mathrm{Bun}_{G,\omega}^{rs})\geq 2$$ (see Remark \ref{rmk codim stable principal} $\&$ \ref{codim regularly stable principal}), and the result follows from Corollary \ref{cor cox ring} and $\text{Lemma \ref{useful lemma 1}}$. Otherwise, if $g\leq 2$, we go to consider the pointed curve $X^+=(C,\{x\}_{x\in I\cup J})$, where $|J|$ is sufficiently large. Let $\omega_J=(\{B_x\}_{x\in I\cup J}, \{a_x\}_{x\in I\cup J})$ be the parabolic data, where $\mathrm{Lie}(a_x)\in \Phi_0^s$ for ${x\in I\cup J}$, and the weight $\{a_x\}_{x\in I\cup J}$ satisfies assumptions (1) and (2) of Proposition \ref{prop codim semistable and stable parabolic}. Note that $\mathrm{Lie}(a_x), x\in I\cup J$ must be nonzero (see Remark \ref{rmk nonzero}). When ${J\gg 0}$, we have $$\mathrm{Codim}(\mathrm{Bun}_{G,\,\omega_J}\setminus \mathrm{Bun}^{rs}_{G,\,\omega_J})\geq 2$$ by Proposition \ref{prop codim semistable and stable parabolic} $\&$ \ref{prop codimension regularly stable}. Then $\text{Remark \ref{rmk g=0,1,2}}$ and $\text{Lemma \ref{useful lemma 1}}$ imply that $R=\mathbb{V}^{\dagger}_{X^+,\mathfrak{g}}$ is finitely generated. Consider the subalgebra $\mathbb{V}^{\dagger}_{X,\mathfrak{g}}:=\bigoplus_{(\alpha_x,\theta)\leq \ell}\mathbb{V}^{\dagger}_{X,\mathfrak{g},\ell,\alpha}$ $$=
\bigoplus_{(\alpha_x,\theta)\leq \ell}\mathbb{V}^{\dagger}_{X,\mathfrak{g},\ell,\alpha\cup \{0\}_{x\in J}}
\subset \underset{(\alpha_x,\theta)\leq \ell}{\bigoplus}\mathbb{V}^{\dagger}_{X^+,\mathfrak{g},\ell,\alpha\cup \{\alpha_x\}_{x\in J}}:=\mathbb{V}^{\dagger}_{X^+,\mathfrak{g}},$$ where the second equality follows from Theorem \ref{theorem Propagation}. 
Set $$\mathbb{Z} \times \prod_{x\in I}X^{\ast}(B_x):=B\subset A:=\mathbb{Z} \times \prod_{x\in I\cup J}X^{\ast}(B_x).$$
The algebra $R$ is $A$-graded (see (\ref{equation graded})), and $\mathbb{V}^{\dagger}_{X,\mathfrak{g}}=R_B$. Now, the result follows from $\text{Lemma \ref{useful lemma 2}}$, since $B\subset A$  is a subgroup of the free abelian group $A$ of finite rank. We are done.
\end{proof}

Let ${\widetilde{C}={\sqcup}_{j \in J} C_j \xrightarrow{\pi} C}$ be the normalization. For each singular point $y \in C$, let $\pi^{-1}(y)=\{p_y,q_y\}$, and identify  $\pi^{-1}(x)$ with $x$ for every smooth point. Set $I_j:=\{x\in I\mid x\in C_j\}$, $I_j^s=\{p_y,q_y\mid p_y,q_y\in C_j\}$, and $I_j^+=I_j\cup I_j^s$.
\goodbreak

Let $S_\ell$ be the set of ${\mu}=\{\mu_x\}_{x \in \cup_{j\in J}I_j^s}$, where each $\mu_x$ is a dominant integral weight satisfying $(\mu_x,\theta) \leq \ell$ and $\mu_{q_y}=\mu_{p_y}^*$ for every singular point $y \in C$. For $j \in J$, with $\alpha=\{\alpha_x\}_{x\in I}$ and ${\mu} \in S_\ell$, define $(\alpha \star {\mu})_j$ to be the assignment of dominant integral weights to $\{x\}_{x\in I_j^+}$ as follows:
$$\text{(1) if $x\in I_j$, assign $\alpha_x$;\,\,\,\,
(2) if $x\in I_j^s$, assign $\mu_x$.} $$
By applying Theorem \ref{factorization theorem} repeatedly, we have
$$
\mathbb{V}_{X,\mathfrak{g},\ell,\alpha}^\dagger \cong \bigoplus_{{\mu} \in S_\ell} \bigotimes_{j \in J} \mathbb{V}_{X^+_j,\mathfrak{g},\ell,({\alpha} \star {\mu})_j}^\dagger,
$$
and there is a homomorphism of algebras: $\mathbb{V}^{\dagger}_{X,\mathfrak{g}}:=$ $${\bigoplus_{(\alpha_x,\theta)\leq \ell}\mathbb{V}_{X,\mathfrak{g},\ell,\alpha}^\dagger \hookrightarrow \bigoplus_{(\alpha_x,\theta)\leq \ell}\,\, \bigoplus_{\mu\in S_{\ell}}\bigotimes_{j\in J}\mathbb{V}^{\dagger}_{X^+_j,\mathfrak{g},\ell,\{\alpha_x\}_{x\in I_j}\cup \{\mu_x\}_{x\in I_j^s}}:=\bigotimes_{j\in J}\mathbb{V}^{\dagger}_{X_j^+,\mathfrak{g}}}$$
(see \cite[Proposition 3.1]{M}),
where $X_j^+=(C_j,\{x\}_{x\in I_j^+})$.
\goodbreak

\begin{proof}[Proof of Theorem \ref{f.g. for singular}: the singular case]
Since $X_j^+$ ($j\in J$) is a smooth pointed curve, the algebra $\mathbb{V}^{\dagger}_{X_j^+,\mathfrak{g}}$ is finitely generated. So is $\bigotimes_{j\in J}\mathbb{V}^{\dagger}_{X_j^+,\mathfrak{g}}$.
Set $B:=$ \begin{equation}\label{equation 7.3}{\left\{ (\{\ell_j\}_{j\in J},\{\alpha_j\}_{j\in J})\middle| \ell_{j_1}=\ell_{j_2}\,\text{and}\,\alpha_{j_1,q_y}=\alpha^{\ast}_{j_2,p_y}\,\text{for}\,j_1,j_2\in J\right\}}.\end{equation}
Clearly, $B\subset A:=\prod_{j\in J}\big(\mathbb{Z}\times \prod_{x\in I^+_j} X^{\ast}(B_x)\big)$
is a subgroup. The algebra $R:=\bigotimes_{j\in J}\mathbb{V}^{\dagger}_{X_j^+,\mathfrak{g}}$ is $A$-graded (see (\ref{equation graded})), and $\mathbb{V}^{\dagger}_{X,\mathfrak{g}}=R_B$.
Then the result follows from $\text{Lemma \ref{useful lemma 2}}$.
\end{proof}

When $X\in \mathcal{M}_{g,n}$ and $\mathrm{Codim}(\mathrm{Bun}_{G,\omega}\setminus \mathrm{Bun}_{G,\omega}^{rs})\geq 2$, 
the group $\mathrm{Cl}(\mathcal{M}_{G,\omega})$ is a free abelian group of finite rank (see (\ref{equation free})). Then by \cite[Corollary 1.2]{EKW}, 
the Cox ring $\mathrm{Cox}(\mathcal{M}_{G,\omega})$ of the normal projective variety $\mathcal{M}_{G,\omega}$ is integrally closed.
\begin{rmk}\label{rmk closed} From the above argument, the {\em algebra of conformal blocks}  $\mathbb{V}^{\dagger}_{X,\mathfrak{g}}$ for $X\in \overline{\mathcal{M}}_{g,n}$  is also integrally closed using Lemma \ref{useful lemma 2}.
\end{rmk}

As an application, we show that, when $g\geq 2$, moduli spaces of semistable \emph{parabolic $G$-bundles}
on smooth curves admit an irreducible normal projective limit for $ X\in \overline{\mathcal{M}}_{g,n}-\mathcal{M}_{g,n}$.

Let $C$ be a smooth curve of genus $g\geq 2$, and $L$ be an ample line bundle on $\mathcal{M}_{G,\omega}$, where $\omega=(\{P_x\}_{x\in I},a)$ with $n=|I|$. Set $\sL:=\pi_G^{\ast}L$, where $\pi_G:\mathrm{Bun}_{G,\omega}^{ss}\rightarrow \mathcal{M}_{G,\omega}$.
Since $$\mathrm{Codim}(\mathrm{Bun}_{G,\omega}\setminus \mathrm{Bun}^{ss}_{G,\omega})\geq 2\,\,\,\,\text{(see Remark \ref{rmk codim stable principal})},$$
the line bundle $\mathcal{L}$ extends uniquely to $\mathrm{Bun}_{G,\omega}$. By abuse of notation, we still write $\mathcal{L}$ for the extension.
This defines a subgroup $\mathbb{Z}\sL:=B\subset A:= \mathrm{Pic}(\mathrm{Bun}_{G,\{B_x\}_{x\in I}})=\mathbb{Z}\times \prod_{x\in I}X^{\ast}(B_x)$ (with Borel subgroups $B_x\subset P_x$, see (\ref{4.1})).

We now show that for smooth $C$, the moduli space $\mathcal{M}_{G,\omega}=\mathrm{Proj}(R_{B})$, where $R=\mathbb{V}^{\dagger}_{X,\mathfrak{g}}$ is graded by $A$ (see (\ref{equation graded})). In fact, $$\aligned R_B=&\bigoplus_{m\in \mathbb{Z}}H^0(\mathrm{Bun}_{G,\{B_x\}_{x\in I}}, \pi^{\ast}\mathcal{L}^m)=\bigoplus_{m\in \mathbb{Z}}H^0(\mathrm{Bun}_{G,\omega}, \mathcal{L}^m)\\ =&\bigoplus_{m\in \mathbb{Z}}H^0(\mathrm{Bun}^{ss}_{G,\omega}, \mathcal{L}^m)
=\bigoplus_{m\in \mathbb{Z}} H^0(\mathcal{M}_{G,\omega},L^m);\endaligned$$ this follows from our choice of $B\subset A$ and the above codimension estimate.
Moreover, multiplication maps coincide
(see the statement under Theorem \ref{thm isomorphism}). Thus, \begin{equation}\label{equation identity}\mathcal{M}_{G,\omega}=\mathrm{Proj}\Big(\bigoplus_{m\in \mathbb{Z}} H^0(\mathcal{M}_{G,\omega},L^m)\Big)=\mathrm{Proj}(R_B)\end{equation} due to the ampleness of $L$.

The \emph{algebra of conformal blocks} $\mathbb{V}^{\dagger}_{X,\mathfrak{g}}$ can be defined relatively on $\overline{\mathcal{M}}_{g,n}$, and it forms a flat sheaf $\mathbb{V}^{\dagger}_{\mathfrak{g}}$ of finitely generated and integrally closed algebras (see Theorem \ref{f.g. for singular} and Remark \ref{rmk closed}), which is $A$-graded (see (\ref{equation graded})).
Set $\mathbb{M}:=\mathbf{Proj}(\mathbb{R}_B)\rightarrow \overline{\mathcal{M}}_{g,n}$, where $\mathbb{R}_B\subset \mathbb{V}^{\dagger}_{\mathfrak{g}}$ is the sheaf of $B$-graded algebras $R_B$ with $R=\mathbb{V}^{\dagger}_{X,\mathfrak{g}}$. By Lemma \ref{useful lemma 2}, the algebra $R_B$ is finitely generated and integrally closed for every ${X\in \overline{\mathcal{M}}_{g,n}}$.  When $X\in \mathcal{M}_{g,n}$,  the fibre $\mathbb{M}_X=\mathrm{Proj}(R_B)=\mathcal{M}_{G,\omega}$ (see (\ref{equation identity})).  Thus, we have

\begin{cor}\label{corollary compactify} Let $G$ be a classical simple and simply connected group,  and $g\geq 2$. Then there is a flat family $\mathbb{M}\rightarrow \overline{\mathcal{M}}_{g,n}$ such that
\begin{itemize}
\item[(1)] for $X=(C,\{x\}_{x\in I})\in \mathcal{M}_{g,n}$, we have $\mathbb{M}_X\cong \mathcal{M}_{G,\omega}$$;$
\item[(2)] the fibre $\mathbb{M}_X$ over $X\in \overline{\mathcal{M}}_{g,n}-\mathcal{M}_{g,n}$ is an irreducible normal projective variety.
\end{itemize}
\end{cor}

Finally, we consider some special parabolic data $\omega$ when $g<2$. Let $\Phi: G\rightarrow \mathrm{SL}(V)\subset \mathrm{GL}(V)$ be the standard representation, and set \begin{equation}\label{equation fundamental alcove subset}\Phi_0^{\ast}:=\{h\in \mathfrak{h}\mid \,\theta_{\mathrm{sl}(V)}(h)< 1, \alpha(h)\geq 0\,\, \forall\, \alpha\in \Delta_+\}\subset \Phi_0^s\end{equation} (see (\ref{equation fundamental alcove})). For $\omega=(\{P_x\}_{x\in I}, \{a_x\}_{x\in I})$ with $\tau_x=\mathrm{Lie}(a_x)\in \Phi_0^{\ast}$, we have $\theta_{\mathrm{sl}(V)}(\tau_x')<1$ (see Section \ref{section fano}). In this case, $$\mathcal{M}_{G,\omega}=\mathfrak{F}^{ss}_{\Phi\text{-}\mathrm{FlPsBun}}//\mathrm{GL(Y)},\,\,\,\, \mathrm{Bun}^{ss}_{G,\omega}=[\mathfrak{F}^{ss}_{\Phi\text{-}\mathrm{FlPsBun}}/\mathrm{GL(Y)}]$$ (see \cite[\S 5.7]{Heinloth}), and there is a finite map $$\psi:\mathcal{M}_{G,\omega}\rightarrow \mathcal{M}_{\mathrm{SL(V)},\widetilde{\omega}},\,\widetilde{\omega}=(\{\widetilde{P}_x\}_{x\in I},\{\widetilde{{a}}_x\}_{x\in I})$$ (see Lemma \ref{lem finite map type C}).
On $\mathrm{Bun}_{\mathrm{SL}(V),\widetilde \omega}$, there exists a line bundle $\Theta_{\mathrm{Bun}_{\mathrm{SL}(V),\widetilde{\omega}}}$ (see (\ref{equation pull-back of theta bundle}))
such that its restriction $\Theta_{\mathrm{Bun}_{\mathrm{SL}(V),\widetilde{\omega}}}|_{
\mathrm{Bun}^{ss}_{G,\widetilde{\omega}}}$ descends to an ample line bundle $\Theta_{\mathcal{M}_{\mathrm{SL}(V),\widetilde{\omega}}}$ on $\mathcal{M}_{\mathrm{SL}(V),\widetilde{\omega}}$. Set $\Theta_{\mathcal{M}_{G,\omega}}:=\psi^{\ast}
\Theta_{\mathcal{M}_{\mathrm{SL}(V),\widetilde{\omega}}}$ and $\Theta_{\mathrm{Bun}_{G,\omega}}:=\bar \Phi^{\ast} \Theta_{\mathrm{Bun}_{\mathrm{SL}(V),\widetilde{\omega}}}$, where $\bar \Phi:\mathrm{Bun}_{G,\omega}\rightarrow \mathrm{Bun}_{\mathrm{SL}(V),\widetilde{\omega}}$ sends $(E,\{P_x\}_{x\in I})$ to the associated bundle $E_{\mathrm{SL}(V)}$ with induced parabolic structure (see Section \ref{section fano}).
Then, 
using a similar argument in the proof of \cite[Proposition 6.5]{KNR}, we have $$H^0(\mathrm{Bun}_{G,\omega}, \Theta^n_{\mathrm{Bun}_{G,\omega}})=
H^0(\mathrm{Bun}^{ss}_{G,\omega},\pi_G^{\ast}\Theta^n_{\mathcal{M}_{G,\omega}}=\Theta^n_{\mathrm{Bun}_{G,\omega}}|_{\mathrm{Bun}^{ss}_{G,\omega}})$$ ($n\geq 0$). Thus, the moduli space is realized as  \begin{equation}\label{equation unique extension}\mathcal{M}_{G,\omega}=\mathrm{Proj}\big(\bigoplus_{n\geq 0}H^0(\mathrm{Bun}_{G,\omega}, \Theta^n_{\mathrm{Bun}_{G,\omega}})  \big).\end{equation} Set $\mathcal{L}:=\Theta_{\mathrm{Bun}_{G,\omega}}$ on $\mathrm{Bun}_{G,\omega}$, and a similar argument to $\text{Corollary \ref{corollary compactify}}$ implies 
\begin{rmk}\label{rmk 5.7} When $g< 2$,
Corollary \ref{corollary compactify} holds for the parabolic data $\omega=(\{P_x\}_{x\in I}, \{{a}_x\}_{x\in I})$ with $\tau_x=\mathrm{Lie}({a}_x)\in \Phi^{\ast}_0$.
\end{rmk}


\bibliographystyle{plain}

\renewcommand\refname{References}

\end{document}